\documentclass[11pt, reqno]{amsart}
\pdfoutput=1
\usepackage{amsmath}
\usepackage{amssymb}
\usepackage{amsthm}
\usepackage{enumerate}
\usepackage{amscd}
\usepackage{graphicx}
\usepackage{tabu}
\usepackage[top=1.50in, bottom=1.50in, left=1.25in, right=1.25in]{geometry}
\usepackage[all]{xy}
\usepackage{longtable}
\usepackage[utf8]{inputenc}
\usepackage{listings}
\usepackage{filecontents}
\usepackage{scrextend}
\usepackage{url}
\usepackage[all]{xy}
\usepackage{stmaryrd}
\usepackage{bbm}
\usepackage{relsize}
\usepackage{hyperref,color}
\usepackage{mathrsfs}
\usepackage{dsfont}

\DeclareMathOperator{\id}{id}

\DeclareMathOperator{\Vect}{Vect}

\DeclareMathOperator{\der}{der}

\DeclareMathOperator{\N}{N}

\DeclareMathOperator{\BD}{BD}
\DeclareMathOperator{\conv}{conv}
\DeclareMathOperator{\Gal}{Gal}

\DeclareMathOperator{\GL}{GL}

\DeclareMathOperator{\Spec}{Spec}

\DeclareMathOperator{\sss}{ss}

\DeclareMathOperator{\Gr}{Gr}
\DeclareMathOperator{\Tr}{Tr}

\DeclareMathOperator{\IC}{IC}

\let\phi\varphi

\theoremstyle{definition}
\newtheorem{defn}{Definition}[section]

\newtheorem{rmrk}[defn]{Remark}

\theoremstyle{plain}
\newtheorem{thm}[defn]{Theorem}
\newtheorem{prop}[defn]{Proposition}
\newtheorem{lem}[defn]{Lemma}
\newtheorem{cor}[defn]{Corollary}

\newcommand\restr[2]{{
  \left.\kern-\nulldelimiterspace 
  #1 
  \vphantom{\big|} 
  \right|_{#2} 
  }}
  
  \newcommand\simtimes{\mathbin{%
    \stackrel{\sim}{\smash{\times}\rule{0pt}{0.8ex}}%
    }}

\makeatletter
\newenvironment{myproof1}[1][\proofname] {\par\pushQED{\qed}\normalfont\topsep6\p@\@plus6\p@\relax\trivlist\item[\hskip\labelsep\bfseries#1\@addpunct{\textbf{.}}]\ignorespaces}{\popQED\endtrivlist\@endpefalse}
\makeatother

\title[Central elements in affine mod $p$ Hecke algebras via Perverse $\mathbb{F}_p$-sheaves]{Central elements in affine mod $p$ Hecke algebras \protect\\ via Perverse $\mathbb{F}_p$-sheaves}
\author{Robert Cass}
\date{\today}
\address{Department of Mathematics, California Institute of Technology}
\email{rcass@caltech.edu}

\numberwithin{equation}{section}

\begin{document}

\begin{abstract}
Let $G$ be a split connected reductive group over a finite field of characteristic $p > 2$ such that $G_\text{der}$ is absolutely almost simple. We give a geometric construction of perverse $\mathbb{F}_p$-sheaves on the Iwahori affine flag variety of $G$ which are central with respect to the convolution product. We deduce an explicit formula for an isomorphism from the spherical mod $p$ Hecke algebra to the center of the Iwahori mod $p$ Hecke algebra. We also give a formula for the central integral Bernstein elements in the Iwahori mod $p$ Hecke algebra. To accomplish these goals we construct a nearby cycles functor for perverse $\mathbb{F}_p$-sheaves and we use Frobenius splitting techniques to prove some properties of this functor. We also prove that certain equal characteristic analogues of local models of Shimura varieties are strongly $F$-regular, and hence they are $F$-rational and have pseudo-rational singularities.
\end{abstract}

\maketitle

\thispagestyle{empty}

\tableofcontents \newpage
\section{Introduction} \label{Intro}
\subsection{Motivation} \label{motsec}
Let $E$ be a local field of characteristic $p > 0$ with ring of integers $\mathcal{O}_E$ and residue field $\mathbb{F}_q$ where $q$ is a power of $p$. Let $G$ be a split connected reductive group defined over $E$, and fix a maximal torus and a Borel subgroup $T \subset B \subset G$. Let $X_*(T)$ be the group of cocharacters of $G$ and let $X_*(T)^+$ be the monoid of dominant cocharacters. For a compact open subgroup $H \subset G(E)$ let $$\mathcal{H}_H := \{f \colon G(E) \rightarrow \mathbb{F}_p \: : \: f \text{ has compact support and is } H \text{ bi-invariant}\}.$$ The multiplication on $\mathcal{H}_H$ is by convolution of functions. The mod $p$ Hecke algebras $\mathcal{H}_H$ are important in the study of smooth admissible representations of $G(E)$.

As $G$ is split we may let  $K = G(\mathcal{O}_E)$. Then Herzig \cite{modpsatake} (and Henniart--Vign\'{e}ras \cite{modpsatake2} for more general coefficients) constructed a mod $p$ Satake isomorphism $$\mathcal{S} \colon \mathcal{H}_{K} \xrightarrow{\sim} \mathbb{F}_p[X_*(T)^+].$$ We note that the mod $p$ Satake isomorphism is related to the usual Satake isomorphism by an integral Satake isomorphism constructed by Zhu in \cite{ZhuSat}. In \cite{modpGr} we constructed a symmetric monoidal category of perverse $\mathbb{F}_p$-sheaves on the affine Grassmannian of $G$ which gives a geometrization of the inverse of $\mathcal{S}$.

Now let $I \subset K$ be the Iwahori subgroup determined by $B$ and let $Z(\mathcal{H}_I)$ be the center of $\mathcal{H}_I$. If $\mathds{1}_K \in \mathcal{H}_I$ is the function which is $1$ on $K$ and $0$ elsewhere then by work of Vign\'{e}ras \cite{Vigneras2005} and Ollivier \cite{Ollivier2014} there is an isomorphism of $\mathbb{F}_p$-algebras $$\mathcal{C} \colon Z(\mathcal{H}_I) \xrightarrow{\sim} \mathcal{H}_K, \quad f \mapsto f* \mathds{1}_K.$$
In this paper we will construct a functor on perverse $\mathbb{F}_p$-sheaves which geometrizes the inverse of $\mathcal{C}$. This allows us to give geometric proofs of certain combinatorial identities in Iwahori mod $p$ Hecke algebras.

This paper is the next step in a project aimed at providing a categorification of the representation theory of affine mod $p$ Hecke algebras. In future joint work with C. P\'{e}pin and T. Schmidt we plan to apply the results in this paper as well as \cite{PS20} to construct a mod $p$ version of Bezrukavnikov's equivalence in \cite{twohecke}, which we expect will have applications to a mod $p$ local Langlands correspondence. In particular, we hope to give a geometric construction of some instances of Grosse-Kl\"{o}nne's functor \cite{GK2016} from supersingular mod $p$ Hecke modules to Galois representations. We refer the reader to the introduction in \cite{modpGr} for more information on the relation between the objects studied in this paper and the $p$-adic Langlands program.

\subsection{Main results} \label{mainsec}
Let $G$ be a connected reductive group over an algebraically closed field $k$ of characteristic $p > 2$ such that $G_\text{der}$ is almost simple. Fix a maximal torus and a Borel subgroup $T \subset B \subset G$. Let $I \subset L^+G$ be the Iwahori group given by the fiber of $B$ under the projection $L^+G \rightarrow G$ (see Section \ref{S2.1} for the definitions). Let $\Gr$ be the affine Grassmannian of $G$, let $\mathcal{F}\ell  $ be the Iwahori affine flag variety, and let $\pi \colon \mathcal{F}\ell \rightarrow \Gr$ be the projection. 

In \cite[\S 6]{modpGr} we defined the categories of equivariant perverse $\mathbb{F}_p$-sheaves $P_{L^+G}(\Gr, \mathbb{F}_p)$ and $P_I(\mathcal{F}\ell, \mathbb{F}_p)$ as well as a convolution product $*$ on $P_{L^+G}(\Gr, \mathbb{F}_p)$. The convolution product on $P_{L^+G}(\Gr, \mathbb{F}_p)$ preserves the full subcategory of semisimple objects $P_{L^+G}(\Gr, \mathbb{F}_p)^{\sss} \subset P_{L^+G}(\Gr, \mathbb{F}_p)$.  In Section \ref{S2.2} we define the convolution product $\mathcal{F}_1^\bullet * \mathcal{F}_2^{\bullet} \in D_c^b(\mathcal{F}\ell_{\text{\'{e}t}}, \mathbb{F}_p)$ of two $I$-equivariant perverse sheaves $\mathcal{F}_1^\bullet$, $\mathcal{F}_2^{\bullet} \in P_I(\mathcal{F}\ell, \mathbb{F}_p)$. 

Let $W$ be the Weyl group of $G$ and let $\tilde{W}$ be the Iwahori--Weyl group of $G(k (\!( t )\!))$. For $\mu \in X_*(T)^+$ let $\Gr_{\leq \mu} \subset \Gr$ be the corresponding reduced $L^+G$-orbit closure, and for $w \in \tilde{W}$ let $\mathcal{F}\ell_w \subset \mathcal{F}\ell$ be the corresponding reduced $I$-orbit closure. Let $\IC_\mu \in P_{L^+G}(\Gr, \mathbb{F}_p)^{\sss}$ be the shifted constant sheaf $\mathbb{F}_p[\dim \Gr_{\leq \mu}]$ supported on $\Gr_{\leq \mu}$ (see Theorem \ref{simpleobjects}). 

Given $\mu \in X_*(T)^+$ let $\text{Adm}(\mu) \subset \tilde{W}$ be the $\mu$-admissible set defined in (\ref{admdef}). Set $$\mathcal{A}(\mu) := \bigcup_{w \in \text{Adm}(\mu)} \mathcal{F}\ell_{w} \subset \mathcal{F} \ell.$$ Let $\mathcal{Z}_{\mathcal{A}(\mu)} \in D_c^b(\mathcal{F}\ell_{\text{\'{e}t}}, \mathbb{F}_p)$ be the shifted constant sheaf $\mathbb{F}_p[\dim \mathcal{A}(\mu)]$ supported on $\mathcal{A}(\mu)$. 

Our main theorem is as follows.

\begin{thm} \label{mainthm}
There exists an exact functor $\mathcal{Z} \colon P_{L^+G}(\Gr, \mathbb{F}_p)^{\sss} \rightarrow P_{I}(\mathcal{F}\ell, \mathbb{F}_p)$ that satisfies:
\leavevmode
\begin{enumerate}[{\normalfont (i)}]
\item For all $\mu \in X_*(T)^+$ there is a canonical isomorphism $$\mathcal{Z}(\IC_\mu) \cong \mathcal{Z}_{\mathcal{A}(\mu)}.$$
\item For all $\mathcal{F}^\bullet \in P_{L^+G}(\Gr, \mathbb{F}_p)^{\sss}$ there is a canonical isomorphism
$$R\pi_!(\mathcal{Z}(\mathcal{F}^\bullet)) \cong \mathcal{F}^\bullet.$$ 
\item For all $\mathcal{F}_1^\bullet \in P_{L^+G}(\Gr, \mathbb{F}_p)^{\sss}$ and $\mathcal{F}_2^{\bullet} \in P_{I}(\mathcal{F}\ell, \mathbb{F}_p)$ the convolution product $$\mathcal{Z}(\mathcal{F}_1^\bullet) * \mathcal{F}_2^{\bullet}$$ is perverse and $I$-equivariant. 
\item For all $\mathcal{F}_1^\bullet \in P_{L^+G}(\Gr, \mathbb{F}_p)^{\sss}$ and $\mathcal{F}_2^{\bullet} \in P_{I}(\mathcal{F}\ell, \mathbb{F}_p)$ there is a canonical isomorphism
$$\mathcal{Z}(\mathcal{F}_1^\bullet) * \mathcal{F}_2^{\bullet} \cong  \mathcal{F}_2^{\bullet} * \mathcal{Z}(\mathcal{F}_1^\bullet).$$
\item For all $\mathcal{F}_1^\bullet$, $\mathcal{F}_2^{\bullet} \in P_{L^+G}(\Gr, \mathbb{F}_p)^{\sss}$ there is a canonical isomorphism
$$\mathcal{Z}(\mathcal{F}_1^\bullet * \mathcal{F}_2^{\bullet}) \cong \mathcal{Z}(\mathcal{F}_1^\bullet) * \mathcal{Z}(\mathcal{F}_2^{\bullet}).$$
\end{enumerate}
\end{thm}

Our method is analogous to the case of $\overline{\mathbb{Q}}_\ell$-coefficients considered in \cite{GaitsgoryCentral} and \cite{ZhuCoherence}, but it involves some new ideas because we use a different notion of the nearby cycles functor which is well suited for our purposes (see Remark \ref{classicnearby} for the usual definition). As in \cite{modpGr}, we exploit subtle properties of the singularities of affine Schubert varieties. In particular, we use Frobenius splitting techniques to verify that our ad-hoc construction of the nearby cycles functor satisfies the necessary properties. This requires us to prove some new results on the $F$-singularities of affine Schubert varieties, which will be explained in Section \ref{Fsec}.

The relevant facts from the theory of $F$-singularities are that if $X$ is an integral $F$-rational variety of dimension $d$ then the shifted constant sheaf $\mathbb{F}_p[d] \in D_c^b(X_{\text{\'{e}t}}, \mathbb{F}_p)$ is a simple perverse sheaf by \cite[1.7]{modpGr}, and that $F$-rational singularities are pseudo-rational by a result of Smith \cite{SmithFrat}. We will combine these facts with a result of Kov{\'a}cs \cite{ratsing} to deduce that our nearby cycles functor commutes with pushforward along birational morphisms between certain $F$-rational varieties.

\begin{rmrk}
The functor $\mathcal{Z}$ in Theorem 1.1 can be defined on the category $P_{L^+G}(\Gr, \mathbb{F}_p)$, but then it is not clear that this functor is exact. This is the reason we restrict to the subcategory $P_{L^+G}(\Gr, \mathbb{F}_p)^{\sss}$. 
\end{rmrk}

\subsection{Applications to mod $p$ Hecke algebras} \label{Heckesec}
Let $G$ be a split connected reductive group defined over a local field $E$ of characteristic $p > 2$ and residue field $\mathbb{F}_q$. Fix a maximal torus and a Borel subgroup $T \subset B \subset G$. We assume that $T$ and $B$ are defined over $\mathbb{F}_q$ and that $G_{\text{der}}$ is absolutely almost simple. Let $K = G(\mathcal{O}_E)$ and let $t$ be a uniformizer of $E$. Then $\mathcal{H}_K$ has a natural basis $\{\mathds{1}_\mu\}$ indexed by the dominant cocharacters $X_*(T)^+$  where $\mathds{1}_\mu$ is the characteristic function of the double coset $K \mu(t) K$. Similarly if $I \subset K$ is the Iwahori subgroup determined by $B$ then $\mathcal{H}_I$ has a basis $\{\mathds{1}_w\}$ indexed by $\tilde{W}$. Let $\mathds{1}_K \in \mathcal{H}_I$ be the function which is $1$ on $K$ and $0$ elsewhere. In Section \ref{S4} we will show that a version of Theorem 1.1 also holds when we view $\Gr$ and $\mathcal{F}\ell$ as ind-schemes over the finite field $\mathbb{F}_q$. Then by applying the function-sheaf correspondence we will derive the following explicit formula for $\mathcal{C}^{-1} \colon \mathcal{H}_K \rightarrow Z(\mathcal{H}_I)$.

\begin{thm} \label{mainapp} Let $\mathcal{C}$ be the isomorphism $Z(\mathcal{H}_I) \rightarrow \mathcal{H}_K$ such that $\mathcal{C}(f) = f * \mathds{1}_K$. Then
$$\mathcal{C}^{-1} \left( \sum_{\lambda \leq \mu} \mathds{1}_{\lambda} \right) = \sum_{w \in \textnormal{Adm}(\mu)} \mathds{1}_{w}.$$
\end{thm}

For $\mu \in X_*(T)^+$ let $t_\mu$ be the element $\mu$ regarded as an element of $\tilde{W}$, and let $\Lambda_\mu \subset X_*(T)$ be the $W$-orbit of $\mu$. In \cite{Vigneras2005}, Vign\'{e}ras  constructed integral Bernstein elements $B(\lambda) \in \mathcal{H}_I$ for $\lambda \in X_*(T)$ and showed that $\{\sum_{\lambda \in \Lambda_\mu} B(\lambda)\}_{\mu \in X_*(T)^+}$ is an $\mathbb{F}_p$-basis for $Z(\mathcal{H}_I)$. Ollivier \cite{Ollivier2014} showed that these Bernstein elements give rise to an isomorphism of $\mathbb{F}_p$-algebras $$\mathcal{B} \colon \mathbb{F}_p[X_*(T)^+] \rightarrow Z(\mathcal{H}_I), \quad \quad \quad \quad \mu \mapsto \sum_{\lambda \in \Lambda_\mu} B(\lambda).$$  Ollivier also showed that $\mathcal{B}$ is compatible with the mod $p$ Satake isomorphism in the sense that $\mathcal{B} = \mathcal{C}^{-1} \circ \mathcal{S}^{-1}$. 

For our last application, we note that by \cite[2.3]{Ollivier2014} the coefficient of $\mathds{1}_w$ appearing in $\sum_{\lambda \in \Lambda_\mu} B(\lambda) \in \mathcal{H}_I$ is $0$ if $w \notin \text{Adm}(\mu)$ and it is $1$ if $w \in \text{Adm}(\mu)$ and $\ell(w) = \ell(t_\mu)$, where $\ell$ is the length function on $\tilde{W}$. Using Theorem \ref{mainapp} we can compute the rest of the coefficients (\emph{cf.} \cite[5.2]{Ollivier2015}).

\begin{cor} \label{maincor}
Let $\mu \in X_*(T)^+$. Then the integral Bernstein elements satisfy $$\sum_{\lambda \in \Lambda_\mu} B(\lambda) = \sum_{w \in \textnormal{Adm}(\mu)}  \mathds{1}_{w}.$$
\end{cor}

\begin{rmrk}
Let $\tilde{I} \subset I$ be the pro-$p$ Sylow subgroup. The integral Bernstein elements are usually defined in the larger Hecke algebra $\mathcal{H}_{\tilde{I}}$. There is a central idempotent $\epsilon_1 \in \mathcal{H}_{\tilde{I}}$ such that $\mathcal{H}_I = \epsilon_1\mathcal{H}_{\tilde{I}}$ (see \cite[2.14]{Ollivier2014}). The Bernstein elements  we are considering in this paper are the images of the Bernstein elements in \cite{Ollivier2014} after multiplication by $\epsilon_1$. The integral Bernstein elements in \cite{Ollivier2014} also depend on a choice of a sign ($\pm$) and a Weyl chamber, but the central integral Bernstein elements $\sum_{\lambda \in \Lambda_\mu} B(\lambda)$ do not depend on these choices by \cite[3.4]{Ollivier2014}.
\end{rmrk}

\subsection{F-singularities of local models} \label{Fsec}
During the course of proving Theorem \ref{mainthm} we will also prove a result about the singularities of equal characteristic analogues of local models of Shimura varieties. Following the notation in Section \ref{Heckesec}, let $\Gr$ be the affine Grassmannian of $G$ viewed as an ind-scheme over $\mathbb{F}_q$. Then for $\mu \in X_*(T)^+$ there is an associated local model $M_{\mu} \rightarrow \Spec(\mathcal{O}_E)$ such that the generic fiber of $M_{\mu}$ is isomorphic to $\Gr_{\leq \mu} \times \Spec(E)$ and the reduced special fiber is isomorphic to $\mathcal{A}(\mu)$ (see Definition \ref{locdef}).

\begin{thm} \label{thm2}
Suppose that $p > 2$ and that $G_{\text{\normalfont der}}$ is absolutely almost simple and simply connected.  Then for any $\mu \in X_*(T)^+$, every local ring in $M_{\mu}$ is strongly $F$-regular, $F$-rational, and has pseudo-rational singularities.  
\end{thm}

In the mixed characteristic case, local models are used to study the \'{e}tale local structure of integral models of Shimura varieties with parahoric level structure. In the equal characteristic case they are related to moduli spaces of shtukas. We refer the reader to \cite{RicharzLocalModels} for more information on local models, where it is also shown that certain local models are Cohen--Macaulay by using Frobenius splittings of global affine Schubert varieties constructed in \cite{ZhuCoherence}.

We prove Theorem \ref{thm2} by combining the same Frobenius splittings in \cite{ZhuCoherence} with our previous results on the global $F$-regularity of affine Schubert varieties in \cite{modpGr}. In fact, we prove that certain global affine Schubert varieties are strongly $F$-regular (Theorem \ref{BDprops}) and then we deduce the local statement in Theorem \ref{thm2}. We also show that Schubert subvarieties in related Beilinson--Drinfeld and convolution Grassmannians are strongly $F$-regular in Theorem \ref{BDF}.
\\
\\
\textbf{Acknowledgments.} It is a pleasure to thank Mark Kisin, C\'{e}dric P\'{e}pin, and Tobias Schmidt for several discussions and comments on an earlier version of this paper. I thank the referee for their careful reading and helpful suggestions which improved the exposition and simplified some of the proofs. I also thank Karol Koziol, Rachel Ollivier, Marie-France Vign\'{e}ras, David Yang, and Xinwen Zhu for their interest and helpful conversations. Parts of this paper were written while the author visited the University of Paris 13 and the University of Rennes 1, and he would like to thank these institutions for their hospitality. This material is based upon work supported by the National Science Foundation Graduate Research Fellowship Program under Grant No. DGE-1144152.

\section{Construction of the functor $\mathcal{Z}$} \label{S2}

\subsection{Local affine Schubert varieties} \label{S2.1}
Let $k$ be a perfect field of characteristic $p > 0$. For a smooth affine group scheme $G$ over the power series ring $k[\![t]\!]$ we define the loop group $LG$ as the functor on $k$-algebras
$$LG \colon R \mapsto G(R (\!( t )\!)).$$ The positive loop group $L^+G$ is the functor
$$L^+G \colon R \mapsto G(R[\![t]\!]).$$
For each integer $n>0$ we also have the $n$th jet group
$$L^nG \colon R \mapsto G(R[t]/t^n).$$

We now specialize to the case where $G$ is a split connected reductive group defined over $k$ (note that $G$ can also be viewed as a constant group scheme over $k[\![t]\!]$). Let $T \subset B \subset G$ be a maximal torus and a Borel subgroup. Let $I \subset L^+G$ be the Iwahori group given by the fiber of $B$ under the projection $L^+G \rightarrow G$. 
The affine Grassmannian is the fpqc-quotient
$\Gr : = LG/L^+G$ and the affine flag variety is the fpqc-quotient $\mathcal{F}\ell: = LG/I.$ Both $\Gr$ and $\mathcal{F}\ell$ are represented by ind-projective $k$-schemes.

The left $L^+G$-orbits in $\Gr$ are indexed by the set of dominant cocharacters $X_*(T)^+$ and the left $I$-orbits in $\mathcal{F}\ell$ are indexed by the Iwahori--Weyl group $\tilde{W}$ of $G(k (\!( t )\!))$. Given $\mu \in X_*(T)^+$ let $\Gr_\mu = L^+G \cdot \mu(t)$ be the corresponding reduced orbit. The reduced closure of $\Gr_\mu$ is denoted $\Gr_{\leq \mu}$, and it is the union of those $\Gr_\lambda$ for $\lambda \leq \mu$. 
For $w \in \tilde{W}$ we define $C(w)$ to be the corresponding reduced $I$-orbit and we denote its reduced closure by $\mathcal{F}\ell_{w}$. The scheme $C(w)$ is isomorphic to $\mathbb{A}_k^{\ell(w)}$. If $\lambda \in X_*(T)$ we denote by $t_\lambda$ the element $\lambda$ viewed as an element of $\tilde{W}$. Note that $\mu - \lambda$ is a sum of positive coroots with non-negative integer coefficients if and only if $t_\lambda \leq t_\mu$ in the Bruhat order on $\tilde{W}$ by \cite[9.4]{ZhuCoherence}, so there is no ambiguity in the choice of order on $X_*(T)^+$. See \cite{ZhuGra} or \cite[\S 5.1]{modpGr} for more details on these affine Schubert varieties. Note that we used the notation $S_w$ in \cite[\S 5.1]{modpGr} instead of $\mathcal{F}\ell_{w}$. 

We now give the definition of the $\mu$-admissible set appearing in Theorem \ref{mainthm}. Given $\mu \in X_*(T)$ let $\Lambda_\mu = W \cdot \mu \subset X_*(T)$ be the orbit of $\mu$ in $X_*(T)$ under the action of the Weyl group $W$. The $\mu$-admissible set is
\begin{equation} \label{admdef} \text{Adm}(\mu) : = \{ w \in \tilde{W} \: | \: w \leq t_{\lambda} \text{ for some } \lambda \in \Lambda_\mu\}.
\end{equation}

By \cite{FaltingsLoop} and \cite{PappasRapoport}, affine Schubert varieties are normal, Cohen--Macaulay, Frobenius split and have rational singularities if $p \nmid |\pi_1(G_{\text{der}})|$. Additionally, we have the following theorem.

\begin{thm}[{\cite[1.4]{modpGr}}]\label{globFreg}
If $p \nmid |\pi_1(G_{\textnormal{der}})|$ the affine Schubert varieties $\Gr_{\leq \mu}$ and $\mathcal{F}\ell_{w}$ are globally $F$-regular, strongly $F$-regular, and $F$-rational. 
\end{thm}

We refer the reader to \cite{SmithGlobally} for the definition of global $F$-regularity and to \cite{Frationalloc} for the definitions of strong $F$-regularity and $F$-rationality. Global $F$-regularity is a property of projective $k$-schemes. Strong $F$-regularity is defined for noetherian rings $R$ of characteristic $p > 0$ that are $F$-finite (meaning $F_*R$ is a finite $R$-module). If $R$ is reduced then $R$ is $F$-finite if and only if $R^{1/p}$ is a finite $R$-module. A finitely generated $k$-algebra is $F$-finite since $k$ is perfect. 

By \cite[3.1 (a)]{TightClosure}, $R$ is strongly $F$-regular if and only if $R_P$ is strongly $F$-regular for every prime ideal $P$, so it makes sense to say a locally noetherian scheme is strongly $F$-regular if all of its local rings are strongly $F$-regular. The property of $F$-rationality is defined for noetherian rings of characteristic $p > 0$, and we say that a locally noetherian scheme is $F$-rational if all of its local rings are $F$-rational. If $R$ is a homomorphic image of a Cohen--Macaulay ring, then $R$ is $F$-rational if and only if all of its local rings are $F$-rational by \cite[4.2 (e)]{Frationalloc}. We have the following chain of implications for projective $k$-schemes (or more generally projective schemes over an $F$-finite field):
$$
\begin{array}{lr} \text{Globally } \\ F \text{-regular} \\ \end{array} \overset{\text{\cite[3.10]{SmithGlobally}}}{\Longrightarrow} \begin{array}{lr} \text{Strongly } \\ F\text{-regular} \\ \end{array} \overset{\text{\cite[3.1]{TightClosure}}}{\Longrightarrow} F\text{-rational} \overset{\begin{array}{lr} \text{\scriptsize \cite[3.1]{SmithFrat}} \\ \text{\scriptsize \cite[4.2]{Frationalloc}} \\ \end{array}}{\Longrightarrow} \begin{array}{lr}
       \text{Pseudo-rational}\\
       \text{singularities, normal,}\\
      \text{Cohen--Macaulay.}\\
     \end{array} 
$$

\begin{rmrk}
We will also use the notion of pseudo-rationality as defined in \cite{ratsing}. Using \cite[1.13]{SmithFrat} and the flat base change theorem, one can verify the following: If $X$ is a $k$-scheme of finite type such that every local ring of $X$ is pseudo-rational as defined in \cite[1.8]{SmithFrat}, then $X$ is also pseudo-rational as defined in \cite[1.2]{ratsing}.
\end{rmrk}

\subsection{Global affine Schubert varieties} \label{S.global} We continue using the notation introduced in Section \ref{S2.1}. Let $C = \mathbb{A}^1_k$. Throughout this paper we will denote by $0$ the origin viewed as a closed point in $C$. Let $\hat{\mathcal{O}}_{0}$ be the completed local ring of $C$ at $0$ and let $C^\circ = C - 0$. Let $\mathcal{G}$ be a Bruhat--Tits group scheme over $C$ equipped with isomorphisms 
\begin{equation} \label{isochoice}
\restr{\mathcal{G}}{C^\circ} \cong G \times C^\circ, \quad L^+(\restr{\mathcal{G}}{{\hat{\mathcal{O}}_{0}}}) \cong I.
\end{equation}
See \cite[\S 3.2]{ZhuCoherence} for more information on the construction of $\mathcal{G}$.

For any smooth group scheme $H$ over $C$ (including $\mathcal{G}$) we let $\mathcal{E}_0$ be $H$ regarded as a trivial $H$-torsor. For a $k$-algebra $R$ let $C_R = C \times_{\Spec(k)} \Spec(R)$. If $x \in C(R)$ let $\Gamma_x \subset C_R$ be the graph of $x$, that is, the closed subscheme $\Spec(R) \xrightarrow{(x, \id)} C \times_{\Spec(k)} \Spec(R)$. The global affine Grassmannian $\Gr_{\mathcal{G}}$ is the functor on $k$-algebras defined by

$$\Gr_{\mathcal{G}}(R) = \left\{  (x, \mathcal{E}, \beta)  \: | \: x \in C(R), \: \mathcal{E} \text{ is a } \mathcal{G}\text{-torsor on } C_R, \:  \beta \colon \restr{ \mathcal{E}}{C_R - \Gamma_x} \cong \restr{ \mathcal{E}_0}{C_R - \Gamma_x} \right\}.$$ 
In the above definition we really mean the set of such objects up to the equivalence relation $(x, \mathcal{E}, \beta) \sim (x, \mathcal{E}', \beta')$ if there is an isomorphism $\mathcal{E} \cong \mathcal{E}'$ which respects the trivializations, but we will suppress this detail. The functor $\Gr_{\mathcal{G}}$ is represented by an ind-projective scheme over $C$ by \cite[5.5]{PappasZhu}. 

Given $x \in C(R)$ let $\hat{\Gamma}_x$ be the formal completion of $C_R$ along $\Gamma_x$ and let $\hat{\Gamma}^\circ_x = \hat{\Gamma}_x - \Gamma_x$. The global analogue of $LG$ is the functor
$$\mathcal{L}\mathcal{G}(R) = \{ (x, \beta) \: : \: x \in C(R), \: \beta \in \mathcal{G}(\hat{\Gamma}^\circ_x) \}.$$
The global analogue of $L^+G$ is the functor
$$\mathcal{L}^+\mathcal{G}(R) = \{ (x, \beta) \: : \: x \in C(R), \: \beta \in \mathcal{G}(\hat{\Gamma}_x) \}.$$
As in \cite[3.1]{ZhuCoherence}, a lemma of Beauville--Laszlo \cite{BeauvilleLaszlo} implies that there is a natural bijection
$$\Gr_{\mathcal{G}}(R) \cong  \left\{  (x, \mathcal{E}, \beta)  \: | \: x \in C(R), \: \mathcal{E} \text{ is a } \mathcal{G}\text{-torsor on } \hat{\Gamma}_x, \:  \beta \colon \restr{ \mathcal{E}}{\hat{\Gamma}^\circ_x} \cong \restr{ \mathcal{E}_0}{\hat{\Gamma}^\circ_x} \right\}.$$
Then $\mathcal{L}^+\mathcal{G}$ acts on $\Gr_{\mathcal{G}}$ by changing the trivialization $\beta$, and there is an isomorphism
$$\mathcal{L}\mathcal{G} / \mathcal{L}^+\mathcal{G} \cong \Gr_{\mathcal{G}}.$$ 
Our choice of isomorphisms in (\ref{isochoice}) induces isomorphisms
\begin{equation} \label{isochoice2} 
\restr{\Gr_{\mathcal{G}}}{C^{\circ}} \cong \Gr \times C^\circ, \quad (\Gr_{\mathcal{G}})_{0} \cong \mathcal{F}\ell
\end{equation} and
\begin{equation} \label{isochoice3} \mathcal{L}^+\restr{\mathcal{G}}{C^{\circ}} \cong L^+G \times C^\circ, \quad (\mathcal{L}^+\mathcal{G})_{0} \cong I. 
\end{equation}
Via the isomorphisms (\ref{isochoice2}) and (\ref{isochoice3}), the action of 
$\mathcal{L}^+ \mathcal{G}$ on $\Gr_{\mathcal{G}}$ is compatible with the action of $L^+G$ on $\Gr$ and the action of $I$ on $\mathcal{F}\ell$.

We define $\overline{\Gr}_{\mathcal{G}, \mu}$ to be the reduced closure of $\Gr_{\leq \mu} \times C^{\circ}$ in $\Gr_{\mathcal{G}}$. The scheme $\overline{\Gr}_{\mathcal{G}, \mu}$ is stable under the action of $\mathcal{L}^+ \mathcal{G}$, and our definition agrees with that in \cite[3.1]{ZhuCoherence} because $G$ is split. We can now define the local model $M_{\mu}$. 

\begin{defn} \label{locdef}
The local model $M_\mu$ is the fiber of $\overline{\Gr}_{\mathcal{G}, \mu}$ over the completed local ring at $0 \in C$. 
\end{defn}
Thus the generic fiber of $M_\mu$ is isomorphic to $\Gr_{\leq \mu} \times \Spec(k(\!( t )\!))$. The following theorem is due to Zhu in the case $G_{\text{der}}$ is absolutely almost simple and simply connected and was extended by Haines--Richarz to the general case. 
\begin{thm}[{\cite[Thm. 3]{ZhuCoherence},  \cite[5.14]{RicharzWeil}, \cite[2.1]{RicharzLocalModels}}] \label{ZhuFiber}
If $p \nmid |\pi_1(G_{\textnormal{der}})|$ the fiber $(\overline{\Gr}_{\mathcal{G}, \mu})_{0}$ is reduced. Without any assumptions on $p$, the reduced fiber satisfies $$(\overline{\Gr}_{\mathcal{G}, \mu})_{0, \, \textnormal{red}} \cong \mathcal{A}(\mu).$$
\end{thm}

For each integer $n > 0$ let $\Gamma_{x,n}$ be the $n$th nilpotent thickening of $\Gamma_x$. The $n$th jet group of $\mathcal{G}$ is
$$\mathcal{L}^+_n \mathcal{G} (R) = \{ (x, \beta) \: : \: x \in C(R), \: \beta \in \mathcal{G}(\Gamma_{x,n}) \}.$$ This functor is represented by a smooth affine group scheme over $C$. For each $\mu \in X_*(T)^+$ the action of $\mathcal{L}^+ \mathcal{G}$ on $\overline{\Gr}_{\mathcal{G}, \mu}$ factors through $\mathcal{L}^+_n \mathcal{G}$ for sufficiently large $n$ depending on $\mu$. If $x \in C^\circ(k)$ then $(\mathcal{L}^+_n \mathcal{G})_x \cong L^n G$, and $(\mathcal{L}^+_n \mathcal{G})_0 \cong L^n(\restr{\mathcal{G}}{{\hat{\mathcal{O}}_{0}}})$.

Finally, let $\underline{G}$ be the constant group scheme $G \times C$. By replacing $\mathcal{G}$ with $\underline{G}$ in the above definitions we get the ind-scheme $\Gr_{\underline{G}}$, which is naturally isomorphic to $\Gr \times C$. There is a natural morphism $\mathcal{G} \rightarrow \underline{G}$ which induces a morphism $\pi_{\mathcal{G}} \colon \Gr_{\mathcal{G}} \rightarrow \Gr_{\underline{G}}$. By taking the fibers of $\pi_{\mathcal{G}}$ over $C^{\circ}$ and $0$ we get the following diagram with Cartesian squares:
$$\xymatrix{
\restr{\Gr_{\mathcal{G}}}{C^{\circ}} \ar[r]^{j_{\mathcal{G}}} \ar[d]^{\sim} & \Gr_{\mathcal{G}} \ar[d]^{\pi_{\mathcal{G}}} & \mathcal{F}\ell \ar[l]_{i_{\mathcal{G}}} \ar[d]^{\pi} \\
\restr{\Gr_{\underline{G}}}{C^\circ} \ar[r]^{j_{\underline{G}}}  & \Gr_{\underline{G}}  & \Gr \ar[l]_{i_{\underline{G}}}  }
$$

\subsection{The definition of $\mathcal{Z}$} \label{S2.2}
For this section we assume that $k$ is an algebraically closed field of characteristic $p > 0$. We refer the reader to \cite[\S 2]{modpGr} for an introduction to the category $P_c^b(X, \mathbb{F}_p)$ of perverse $\mathbb{F}_p$-sheaves on a separated scheme $X$ of finite type over $k$. This is an abelian subcategory of $D_c^b(X_{\text{\'{e}t}}, \mathbb{F}_p)$ in which every object has finite length. As in the case of perverse $\overline{\mathbb{Q}}_\ell$-sheaves, there are operations such as the intermediate extension functor and pullback along smooth morphisms.

Now suppose $X$ is a separated scheme of finite type over $C$. Let $j \colon U \rightarrow X$ be the inclusion of the fiber of $X$ over $C^\circ$ and let $i \colon Z \rightarrow X$ be the inclusion of the fiber of $X$ over $0$. For $\mathcal{F}^\bullet \in P_c^b(U, \mathbb{F}_p)$ we define the nearby cycles of $\mathcal{F}^\bullet$ by
$$\Psi_X(\mathcal{F}^\bullet) := Ri^*(j_{!*}\mathcal{F}^\bullet)[-1].$$
This defines an additive functor $$\Psi_X \colon P_c^b(U, \mathbb{F}_p) \rightarrow D_c^b(Z_{\text{\'{e}t}}, \mathbb{F}_p).$$

\begin{prop} \label{nearbysmooth}
Suppose $f \colon X' \rightarrow X$ is a smooth, separated  morphism over $C$ of relative dimension $d$ and that $X'$ has fibers $U'$ and $Z'$ over $C^\circ$ and $0$, respectively. Then there is a natural isomorphism of functors
$$R(\restr{f}{Z'})^*[d] \circ \Psi_X \cong \Psi_{X'} \circ R(\restr{f}{U'})^*[d] \colon P_c^b(U, \mathbb{F}_p) \rightarrow D_c^b(Z'_{\textnormal{\'{e}t}}, \mathbb{F}_p).$$
\end{prop}

\begin{proof}
This follows from the fact that pullback along a smooth morphism commutes with taking intermediate extensions by \cite[2.16]{modpGr}. 
\end{proof}

\begin{rmrk} If $f \colon X' \rightarrow X$ is a proper morphism then the functor $Rf_!$ does not preserve perversity in general. However, if $\mathcal{F}^\bullet \in P_c^b(U', \mathbb{F}_p)$ and $R(\restr{f}{U'})_!(\mathcal{F}^\bullet)$ happens to be perverse, then it makes sense to ask whether $R(\restr{f}{Z'})_!(\Psi_{X'} (\mathcal{F}^\bullet))$ and $\Psi_{X}(R(\restr{f}{U'})_!(\mathcal{F}^\bullet))$ are isomorphic. We will see examples (such as the proof of Theorem \ref{mainthm} (ii)) where the presence of $F$-rational singularities allows us to prove there is such an isomorphism, but we are unsure about the general case. By \cite[Exp. XIII 2.1.7.1]{SGA7II}, the analogous fact is true for the usual definition of the nearby cycles functor (see Remark \ref{classicnearby}) due to the proper base change theorem.
\end{rmrk}

\begin{rmrk} We note that by \cite[2.7 (ii)]{modpGr}, $\Psi_X(\mathcal{F}^\bullet) \in {}^pD^{\leq 0}({Z_\text{\'{e}t}}, \mathbb{F}_p)$. By \cite[4.4.2]{BBD}, the same is true for $\mathbb{F}_{\ell}$-sheaves with $\ell \neq p$ using the usual definition of the nearby cycles functor. As the usual nearby cycles functor commutes with Verdier duality for $\mathbb{F}_{\ell}$-sheaves by \cite[4.2]{Illusiemonodromie}, then it preserves perversity. There is no duality functor for $\mathbb{F}_p$-sheaves, and we do not know if $\Psi_X(\mathcal{F}^\bullet)$ is always perverse, but it is perverse in all of the examples we have computed.
\end{rmrk}

By the following lemma, we can naturally extend the definition of the nearby cycles functor to the ind-scheme $\Gr_{\mathcal{G}}$.

\begin{lem} \label{nearbyind}
Suppose $h \colon X' \rightarrow X$ is a closed immersion of separated $C$-schemes of finite type and that $X'$ has fibers $U'$ and $Z'$ over $C^\circ$ and $0$, respectively. Then there is a natural isomorphism of functors
$$ R(\restr{h}{Z'})_* \circ \Psi_{X'} \cong \Psi_X \circ R(\restr{h}{U'})_* \colon P_c^b(U', \mathbb{F}_p) \rightarrow D_c^b(Z_{\textnormal{\'{e}t}}, \mathbb{F}_p).$$ 
\end{lem}

\begin{proof}
By taking the fibers of $h$ over $C^\circ$ and $0$ we get the following diagram with Cartesian squares:
$$\xymatrix{
U' \ar[r]^{j'} \ar[d] & X' \ar[d]^{h} & Z' \ar[l]_{i'} \ar[d] \\
U \ar[r]^{j}  & X  & Z \ar[l]_{i}  }
$$
Because the intermediate extension functor agrees with the usual pushforward functor for closed immersions then by \cite[2.6]{modpGr} we have
$$Rh_* \circ j_{!*}' \cong j_{!*} \circ R(\restr{h}{U'})_*.$$ Now the lemma follows by applying $Ri^*$ and the proper base change theorem.
\end{proof}

\begin{rmrk} \label{classicnearby}
Following \cite[Exp. XIII \S 1.3]{SGA7II}, there is another nearby cycles functor $\Psi_X'$ defined as follows. Replace $C$ with the henselization of its local ring at $0$. Let $\overline{\eta}$ be the spectrum of an algebraic closure of the function field of $C$, and let $\overline{C}$ be the normalization of $C$ in $\overline{\eta}$. Let $X_{\overline{C}} = X \times_C \overline{C}$ and $X_{\overline{\eta}} = X \times_C \overline{\eta}$. Then there are natural morphisms $\overline{j} \colon X_{\overline{\eta}} \rightarrow X_{\overline{C}}$ and $\overline{i} \colon Z \rightarrow X_{\overline{C}}$. If $\restr{\mathcal{F}^\bullet}{\overline{\eta}}$ is the restriction of $\mathcal{F}^\bullet \in D^b_c(U_{\text{\'{e}t}}, \mathbb{F}_p)$ to $X_{\overline{\eta}}$, one can define
$$\Psi'_X(\mathcal{F}^\bullet) := R \overline{i}^* R\overline{j}_* (\restr{\mathcal{F}^\bullet}{\overline{\eta}}).$$ For $\mathbb{F}_{\ell}$-sheaves with $\ell \neq p$, the functor $\Psi_X'$ preserves constructibility by \cite[Exp. 7, Th. 3.2]{SGA412} and commutes with smooth base change as in Proposition \ref{nearbysmooth} by \cite[Exp. XIII 2.1.7]{SGA7II}. The proofs make essential use of the hypothesis $\ell \neq p$, and part of our original motivation for defining $\Psi_X$ was to give short proofs of these two facts for $\mathbb{F}_p$-sheaves. It would be interesting to compare the functors $\Psi_X$ and $\Psi_X'$ for $\mathbb{F}_p$-sheaves.
\end{rmrk}

In \cite[\S 6]{modpGr} we defined the category $P_{L^+G}(\Gr, \mathbb{F}_p)$ of $L^+G$-equivariant perverse $\mathbb{F}_p$-sheaves on $\Gr$. We also defined the category $P_I(\mathcal{F} \ell, \mathbb{F}_p)$ and proved the following theorem.

\begin{thm}[{\cite[1.5]{modpGr}}] \label{simpleobjects}
The simple objects in $P_{L^+G}(\Gr, \mathbb{F}_p)$ and $P_I(\mathcal{F} \ell, \mathbb{F}_p)$ are the shifted constant sheaves $$\IC_\mu : = \mathbb{F}_p[\dim \Gr_{\leq \mu}] \in D_c^b(\Gr_{\leq \mu, \textnormal{\'{e}t}}, \mathbb{F}_p), \quad \quad \quad \quad \IC_{w}^{\mathcal{F}\ell} := \mathbb{F}_p[\dim \mathcal{F}\ell_w] \in D_c^b(\mathcal{F}\ell_{w, \textnormal{\'{e}t}}, \mathbb{F}_p),$$ for $\mu \in X_*(T)^+$, $w \in \tilde{W}$. 
\end{thm}

Let $\mathcal{F}^\bullet \in P_{L^+G}(\Gr, \mathbb{F}_p)^{\sss}$. Since $C$ is smooth then $\mathcal{F}^\bullet \overset{L}{\boxtimes} \mathbb{F}_p[1] \in P_c^b(\Gr \times C, \mathbb{F}_p)$ by \cite[2.15]{modpGr}. Via the isomorphism $\restr{\Gr_{\mathcal{G}}}{C^{\circ}} \cong \Gr \times C^\circ$ we view $\mathcal{F}^\bullet \overset{L}{\boxtimes} \restr{\mathbb{F}_p[1]}{C^\circ}$ as a perverse sheaf on $\restr{\Gr_{\mathcal{G}}}{C^{\circ}}$. We set $$\mathcal{Z}(\mathcal{F}^\bullet) := \Psi_{\Gr_{\mathcal{G}}}(\mathcal{F}^\bullet \overset{L}{\boxtimes} \restr{\mathbb{F}_p[1]}{C^\circ}) \in D_c^b(\mathcal{F}\ell_{\text{\'{e}t}}, \mathbb{F}_p).$$

In \cite[\S 6]{modpGr} we defined the convolution product $\mathcal{F}_1^\bullet * \mathcal{F}_2^{\bullet}$ of $\mathcal{F}_1^\bullet$, $\mathcal{F}_2^{\bullet} \in P_{L^+G}(\Gr, \mathbb{F}_p)$. We now define the convolution product of two perverse sheaves in $P_I(\mathcal{F}\ell, \mathbb{F}_p)$. Since the situation is analogous to $P_{L^+G}(\Gr, \mathbb{F}_p)$ we will be brief. To begin, we have the convolution diagram
\begin{equation}\ \label{convdiagram}
\mathcal{F}\ell \times \mathcal{F}\ell \xleftarrow{p} LG \times \mathcal{F}\ell \xrightarrow{q} LG \times^{I} \mathcal{F}\ell \xrightarrow{m} \mathcal{F}\ell. 
\end{equation} Here $p$ is the quotient map $LG \rightarrow \mathcal{F}\ell$ on the first factor and the identity map on the second factor. The map $q$ is the quotient by the diagonal action of $I$ given by $g \cdot (g_1, g_2) = (g_1 g^{-1}, g g_2)$, and $m$ is the multiplication map. We will also use the notation $\mathcal{F}\ell \simtimes \mathcal{F}\ell$ for $LG \times^I \mathcal{F}\ell$.

Given $\mathcal{F}_1^\bullet$, $\mathcal{F}_2^{\bullet} \in P_I(\mathcal{F} \ell, \mathbb{F}_p)$ we claim that there is a unique perverse sheaf $\mathcal{F}_1^\bullet \overset{\sim}{\boxtimes} \mathcal{F}_2^{\bullet} \in P_I(LG \times^{I} \mathcal{F}\ell , \mathbb{F}_p)$ such that $Rp^*(\mathcal{F}_1^\bullet \overset{L}{\boxtimes} \mathcal{F}_2^{\bullet} ) \cong Rq^*(\mathcal{F}_1^\bullet \overset{\sim}{\boxtimes} \mathcal{F}_2^{\bullet} ).$ The proof of this is analogous to the case of the affine Grassmannian in \cite[6.2]{modpGr} so we omit it. We are also suppressing the fact that because $LG \times \mathcal{F}\ell$ is not of ind-finite type, we must replace the $I$-torsors $p$ and $q$ by torsors for a finite type quotient of $I$ depending on the support of $\mathcal{F}_1^\bullet$ and $\mathcal{F}_2^{\bullet}$. 

The convolution of $\mathcal{F}_1^\bullet$ and $\mathcal{F}_2^{\bullet}$ is $$\mathcal{F}_1^\bullet * \mathcal{F}_2^{\bullet} = Rm_!(\mathcal{F}_1^\bullet \overset{\sim}{\boxtimes} \mathcal{F}_2^{\bullet}) \in D_c^b(\mathcal{F}\ell_{\text{\'{e}t}}, \mathbb{F}_p).$$ We may also write $Rm_*$ instead of $Rm_!$ because $\mathcal{F}_1^\bullet \overset{\sim}{\boxtimes} \mathcal{F}_2^{\bullet}$ is supported on a proper scheme. As in the case of $\overline{\mathbb{Q}} _\ell$-coefficients, $\mathcal{F}_1^\bullet * \mathcal{F}_2^{\bullet}$ is not perverse in general. However, if $\mathcal{F}_1^\bullet * \mathcal{F}_2^{\bullet}$ is perverse then it is also $I$-equivariant by \cite[3.2]{modpGr}. 

\begin{rmrk}
Using the method in \cite[3.13]{modpGr} we can define the category $P_{\mathcal{L}^+\mathcal{G}}(\Gr_{\mathcal{G}}, \mathbb{F}_p)$ of $\mathcal{L}^+\mathcal{G}$-equivariant perverse $\mathbb{F}_p$-sheaves on $\Gr_{\mathcal{G}}$. By the same reasoning we can define other categories of equivariant perverse sheaves on ind-schemes we introduce later, such as $P_{\mathcal{L}^+\mathcal{G}}(\Gr_{\mathcal{G}}^{\conv}, \mathbb{F}_p)$ (see Section \ref{BDsec}). As $\mathcal{L}^+_n \mathcal{G} \rightarrow C$ has geometrically connected fibers for every $n$ then $P_{\mathcal{L}^+\mathcal{G}}(\Gr_{\mathcal{G}}, \mathbb{F}_p)$ is a full subcategory of $P_c^b(\Gr_{\mathcal{G}}, \mathbb{F}_p)$.
\end{rmrk}

\subsection{First properties of $\mathcal{Z}$}
In this section we prove parts (i) and (ii) of Theorem \ref{mainthm}. The main ingredient will be the $F$-rationality of $\overline{\Gr}_{\mathcal{G}, \mu}$. We assume that $k$ is a perfect field of characteristic $p> 2$ until Proposition \ref{constprop}, where we require $k$ to be algebraically closed. Throughout this section we also assume that $G_{\text{der}}$ is absolutely almost simple and simply connected. We will explain how to remove the simple connectedness hypothesis in Remark \ref{assumption}. To begin, we recall the following results.

\begin{thm}[{\cite[9.1]{PappasZhu}, \cite[2.1]{RicharzLocalModels}}] \label{thm2.1}
The schemes $\overline{\Gr}_{\underline{G}, \mu}$ and $\overline{\Gr}_{\mathcal{G}, \mu}$ are integral, normal, and Cohen--Macaulay.
\end{thm}

\begin{cor} \label{fibercor}
The fiber $(\overline{\Gr}_{\mathcal{G}, \mu})_{0}$ is Cohen--Macaulay, connected, and equidimensional of dimension equal to that of $\Gr_{\leq \mu}$. 
\end{cor}

\begin{proof}
Since $\overline{\Gr}_{\mathcal{G}, \mu}$ is Cohen--Macaulay and integral then $(\overline{\Gr}_{\mathcal{G}, \mu})_{0} $ is also Cohen--Macaulay by \cite[OC6G]{stacks-project}. Note that the morphism $\overline{\Gr}_{\mathcal{G}, \mu} \rightarrow C$ is flat by \cite[III 9.7]{HartshorneAG}.  Now because the generic fiber of $\overline{\Gr}_{\mathcal{G}, \mu} \rightarrow C$ is connected, then so is $(\overline{\Gr}_{\mathcal{G}, \mu})_{0} $ by \cite[15.5.4]{EGA4III}. Finally, $(\overline{\Gr}_{\mathcal{G}, \mu})_{0}$ is equidimensional because it is Cohen--Macaulay and connected, and $\dim \,(\overline{\Gr}_{\mathcal{G}, \mu})_{0} = \dim \Gr_{\leq \mu}$ by \cite[III 9.6]{HartshorneAG}.
\end{proof}

The following lemmas (\ref{prodlem}, \ref{prodlem2}, \ref{Fregdescent}) are well known to experts, but because we could not find detailed proofs in the literature we provide them here. These three lemmas are also valid if $p=2$. 

\begin{lem} \label{prodlem}
Let $A$ and $B$ be domains that are strongly $F$-regular $k$-algebras of finite type. Then if $A \otimes_k B$ is a domain it is strongly $F$-regular.
\end{lem}

\begin{proof}
This is proven in \cite[5.2]{ProdSegre} when $A$ and $B$ are graded rings, but the same proof works in general. Let $R = A \otimes_k B$ and let $a \in A$ and $b \in B$ be such that the localizations $A_a$ and $B_b$ are smooth. Then $A_a \otimes_k B_b$ is smooth and hence strongly $F$-regular by \cite[3.1 (c)]{TightClosure}, so by \cite[3.3 (a)]{TightClosure} it suffices to construct a splitting of $R [(a \otimes b)^{1/q}] \subset R^{1/q}$ for some $q = p^e$. Because $k$ is perfect, such a splitting can be constructed from splittings of $A[a^{1/q}] \subset A^{1/q}$ and $B[b^{1/q}] \subset B^{1/q}$, which exist for some common $q$ because $A_a$ and $B_b$ are strongly $F$-regular.
\end{proof}

\begin{lem} \label{prodlem2}
Let $R$ be a domain that is a strongly $F$-regular $k$-algebra of finite type and let $E$ be an $F$-finite field containing $k$. Then if $R \otimes_k E$ is a domain it is strongly $F$-regular.
\end{lem}

\begin{proof}
Let $R_E = R \otimes_k E$ and let $c \in R$ be such that the localization $R_c$ is a smooth $k$-algebra. Since $R_c \otimes_k E$ is a smooth $E$-algebra, it is regular. Hence by \cite[3.1 (c)]{TightClosure}, $R_c \otimes_k E$ is strongly $F$-regular. Thus, to prove $R_E$ is strongly $F$-regular, by \cite[3.3 (a)]{TightClosure} it suffices to show that the inclusion $R_E[c^{1/q} \otimes 1] \subset R_E^{1/q}$ splits for some $q = p^e$. Because $k$ is perfect, then such a splitting can be constructed from splittings of $R[c^{1/q}] \subset R^{1/q}$ and $E \subset E^{1/q}$. A splitting of  $R[c^{1/q}] \subset R^{1/q}$ exists for some $q$ because $R$ is strongly $F$-regular, and a splitting of  $E \subset E^{1/q}$ exists because $E$ is $F$-finite and it is a field.
\end{proof}

\begin{lem} \label{Fregdescent}
Let $f \colon Y \rightarrow X$ be a smooth surjective morphism between reduced $k$-schemes of finite type. Then $X$ is strongly $F$-regular if and only if $Y$ is strongly $F$-regular. 
\end{lem}

\begin{proof}
If $Y$ is strongly $F$-regular then $X$ is strongly $F$-regular by \cite[3.1 (b)]{TightClosure}. Conversely, if $X$ is strongly $F$-regular then $Y$ is strongly $F$-regular by \cite[3.6]{Aberbach2001}. In order to apply \cite[3.6]{Aberbach2001} we are using the fact that the local rings of a smooth scheme over a field are regular, and hence they are Gorenstein and $F$-rational by \cite[3.4]{Frationalloc}. 
\end{proof}

We can now prove that the global Schubert varieties are strongly $F$-regular. 

\begin{thm} \label{BDprops}
The schemes $\overline{\Gr}_{\underline{G}, \mu}$ and $\overline{\Gr}_{\mathcal{G}, \mu}$ are strongly $F$-regular, $F$-rational, and have pseudo-rational singularities.  
\end{thm}

\begin{proof}
By the implications following Theorem \ref{globFreg} it suffices to prove these schemes are strongly $F$-regular. By \cite[1.4]{modpGr}, $\Gr_{\leq \mu}$ is globally $F$-regular and hence also strongly $F$-regular. As $C$ is smooth then it is strongly $F$-regular by \cite[3.1 (c)]{TightClosure}. Thus since $\overline{\Gr}_{\underline{G}, \mu} \cong \Gr_{\leq \mu} \times C$ then $\overline{\Gr}_{\underline{G}, \mu}$ is strongly $F$-regular by Theorem \ref{thm2.1} and Lemma \ref{prodlem}. 

Since $\restr{\Gr_{\mathcal{G}}}{C^{\circ}} \cong \restr{\Gr_{\underline{G}}}{C^{\circ}}$ then $\restr{\Gr_{\mathcal{G}}}{C^{\circ}}$ is strongly $F$-regular. Hence by \cite[3.3 (a)]{TightClosure}, to prove $\overline{\Gr}_{\mathcal{G}, \mu}$ is strongly $F$-regular it suffices to prove $\overline{\Gr}_{\mathcal{G}, \mu}$ is Frobenius split along the effective Cartier divisor $(\overline{\Gr}_{\mathcal{G}, \mu})_{0}$. As $\overline{\Gr}_{\mathcal{G}, \mu}$ is Frobenius split compatibly with $(\overline{\Gr}_{\mathcal{G}, \mu})_{0}$ by \cite[6.5]{ZhuCoherence} then it is also Frobenius split along $(\overline{\Gr}_{\mathcal{G}, \mu})_{0}$ (see, for example, \cite[5.7]{modpGr}). 
\end{proof}

\begin{rmrk} We expect that by essentially the same proof, Theorem \ref{BDprops} is true more generally for any parahoric subgroup of a tamely
ramified connected reductive group $G$ over $k(\!(t)\!)$ such that $p > 2$ and $G_{\text{der}} $ is absolutely almost simple and simply connected. This is because the necessary inputs from \cite{PappasRapoport} and \cite{ZhuCoherence} are proved under these more general assumptions.
\end{rmrk}

Before proceeding we prove Theorem \ref{thm2}. Recall that in the setup of Theorem \ref{thm2}, $E$ is a local field of characteristic $p > 2$ with ring of integers $\mathcal{O}_E$ and residue field $\mathbb{F}_q$. The group $G$ is a split connected reductive group defined over $E$ such that $G_{\text{der}}$ is absolutely almost simple and simply connected. We assume that $G$ is the base extension of a split Chevalley group over $\mathbb{Z}$ and that $T$ and $B$ are defined over $\mathbb{F}_q$. Let $\Gr$ be the affine Grassmannian of $G$ viewed as a group scheme over $\mathbb{F}_q$. After choosing an isomorphism $\mathbb{F}_q  [\![ t ]\!] \cong \mathcal{O}_E$ we can view $M_{\mu}$ as a projective $\mathcal{O}_E$-scheme.

\begin{myproof1}[{Proof of Theorem \ref{thm2}}] Let $x \in M_{\mu}$ and let $\mathcal{O}_x$ be the local ring at $x$. As $\mathcal{O}_E$ is an excellent ring and $M_{\mu}$ is a projective $\mathcal{O}_E$-scheme then $\mathcal{O}_x$ is excellent. Hence by \cite[3.1]{SmithFrat} pseudo-rationality will follow if we prove that $\mathcal{O}_x$ is $F$-rational. The residue field of $\mathcal{O}_x$ is $F$-finite because it is finitely generated as a field over the perfect field $k$. Thus $\mathcal{O}_x$ is $F$-finite by \cite[2.6]{Kunz1976}. By \cite[3.1 (d)]{TightClosure} it suffices to show $\mathcal{O}_x$ is strongly $F$-regular.

First suppose $x \in M_{\mu}$ lies in the closed fiber of $M_\mu \rightarrow \Spec(\mathcal{O}_E)$. The completion $\hat{\mathcal{O}}_x$ is excellent and has an $F$-finite residue field, so it is $F$-finite by \cite[2.6]{Kunz1976}. To prove $\mathcal{O}_x$ is strongly $F$-regular it suffices to prove $\hat{\mathcal{O}}_x$ is strongly $F$-regular by \cite[3.1 (b)]{TightClosure}. The ring $\hat{\mathcal{O}}_x$ is isomorphic to the completion of a local ring in $\overline{\Gr}_{\mathcal{G}, \mu}$. Thus, it suffices to take $y \in \overline{\Gr}_{\mathcal{G}, \mu}$ and show that the completion $\hat{\mathcal{O}}_y$ of the local ring $\mathcal{O}_y$ is strongly $F$-regular. Note that the map $\mathcal{O}_y \rightarrow \hat{\mathcal{O}}_y$ is a flat map between $F$-finite noetherian local rings. Since $\mathcal{O}_y$ is reduced and excellent then $\hat{\mathcal{O}}_y$ is also reduced. Moreover, the fibers of this map are regular by \cite[7.8.3 (v)]{EGA4II}. Thus, since $\mathcal{O}_y$ is strongly $F$-regular (Theorem \ref{BDprops}) then $\hat{\mathcal{O}}_y$ is strongly $F$-regular by \cite[3.6]{Aberbach2001}. This shows that $\mathcal{O}_x$ is strongly $F$-regular if $x$ lies in the closed fiber of $M_\mu \rightarrow \Spec(\mathcal{O}_E)$.

To complete the proof we need to show that the generic fiber $M_{\mu} \times_{\Spec(\mathcal{O}_E)} \Spec(E) = \Gr_{\mathbb{F}_q, \leq \mu} \times_{\Spec(\mathbb{F}_q)} \Spec(E)$ is strongly $F$-regular. The property of strong $F$-regularity is Zariski local by definition, so it suffices to take an arbitrary open affine $\Spec(R) \subset \Gr_{\mathbb{F}_q, \leq \mu}$ and show that $R \otimes_{\mathbb{F}_q} E$ is strongly $F$-regular. Thus by Lemma \ref{prodlem2} we are reduced to showing that $\Gr_{\mathbb{F}_q, \leq \mu}$ is geometrically integral.

We first show that the orbit $\Gr_{\mathbb{F}_q, \mu}$ is geometrically integral. Let $F/\mathbb{F}_q$ be a field extension and let $G_F = G \times_{\Spec(\mathbb{F}_q)} \Spec(F)$. By definition $\Gr_{F, \mu}$ is the scheme-theoretic image of the orbit map $L^n(G_F) \rightarrow \Gr_{F}$, $g \mapsto g \cdot \mu(t)$ for $n \gg 0$. Since $L^nG$ is geometrically integral and taking scheme theoretic images commutes with flat base change it follows that $ \Gr_{\mathbb{F}_q, \mu} \times_{\Spec(\mathbb{F}_q)} \Spec(F) \cong \Gr_{F, \mu}$ is integral. 

Now assume $F$ is an algebraic closure of $\mathbb{F}_q$. Because the Cartan decompositions of $G(\mathbb{F}_q(\!(t)\!))$ and $G(F(\!(t)\!))$ are both indexed by $X_*(T)^+$, it follows that the reduced subscheme of $\Gr_{\mathbb{F}_q, \leq \mu} \times_{\Spec(\mathbb{F}_q)} \Spec(F)$ is isomorphic to $\Gr_{F, \leq \mu}$. Thus $\Gr_{\mathbb{F}_q, \leq \mu}$ is geometrically irreducible by \cite[038I]{stacks-project}. Finally, since $\mathbb{F}_q$ is perfect then $\Gr_{\mathbb{F}_q, \leq \mu}$ is geometrically reduced and hence also geometrically integral.
\end{myproof1}

We now begin proving parts (i) and (ii) of Theorem \ref{mainthm}. For the rest of this section we assume that $k$ is an algebraically closed field of characteristic $p > 2$ and that $G$ is a connected reductive group over $k$ such that $G_{\text{der}}$ is almost simple and simply connected.

\begin{prop} \label{constprop}
The perverse sheaf $j_{\mathcal{G}, !*} (\IC_\mu \overset{L}{\boxtimes} \, \restr{\mathbb{F}_p[1]}{C^\circ})$ is isomorphic to the shifted constant sheaf $\mathbb{F}_p[\dim \overline{\Gr}_{\mathcal{G}, \mu}]$ supported on $ \overline{\Gr}_{\mathcal{G}, \mu}$.
\end{prop}

\begin{proof}
Since $ \overline{\Gr}_{\mathcal{G}, \mu}$ is integral and $F$-rational then $\mathbb{F}_p[\dim \overline{\Gr}_{\mathcal{G}, \mu}]$ is a simple perverse sheaf by \cite[1.7]{modpGr}. Now the lemma follows by applying \emph{loc. cit.} on $ \restr{\overline{\Gr}_{\mathcal{G}, \mu}}{C^\circ}$ and using the fact that $j_{\mathcal{G}, !*} (\IC_\mu \overset{L}{\boxtimes} \, \restr{\mathbb{F}_p[1]}{C^\circ})$ is simple by \cite[2.9]{modpGr}.
\end{proof}

\begin{prop}
The functor $\mathcal{Z}$ is exact and preserves perversity, and $\mathcal{Z}(\IC_\mu) \cong \mathcal{Z}_{\mathcal{A}(\mu)}$.
\end{prop}

\begin{proof}
The isomorphism $\mathcal{Z}(\IC_\mu) \cong \mathcal{Z}_{\mathcal{A}(\mu)}$ follows from Theorem \ref{ZhuFiber} and Proposition \ref{constprop}, and the fact that $\dim \mathcal{A}(\mu) = \dim \Gr_{\leq \mu}$. As $\mathcal{A}(\mu)$ is Cohen--Macaulay and equidimensional (Corollary \ref{fibercor}) then $\mathcal{Z}_{\mathcal{A}(\mu)}$ is perverse by \cite[1.6]{modpGr}. Since $P_{L^+G}(\Gr, \mathbb{F}_p)^{\sss}$ is semisimple then $\mathcal{Z}$ is also exact and it preserves perversity.
\end{proof}

\begin{prop} \label{equivprop}
$\mathcal{Z}(\mathcal{F}^\bullet)$ is $I$-equivariant. 
\end{prop}

\begin{proof}
Note that $\mathcal{F}^\bullet \overset{L}{\boxtimes} \restr{\mathbb{F}_p[1]}{C^\circ}$ is equivariant for the action of $\restr{\mathcal{L}^+ \mathcal{G}}{C^\circ} = L^+G \times C^\circ$ on $\restr{\overline{\Gr}_{\mathcal{G}, \mu}}{C^\circ} = \Gr_{\leq \mu} \times C^\circ$. Using the fact that taking intermediate extensions commutes with smooth pullback (\cite[2.16]{modpGr}) one can check that the perverse sheaf $j_{\mathcal{G}, !*}(\mathcal{F}^\bullet \overset{L}{\boxtimes} \restr{\mathbb{F}_p[1]}{C^\circ}) \in P_c^b(\overline{\Gr}_{\mathcal{G}, {\mu}}, \mathbb{F}_p)$ is $\mathcal{L}^+ \mathcal{G}$-equivariant. By taking the fiber over $0$ it follows that $\mathcal{Z}(\mathcal{F}^\bullet)$ is $I$-equivariant. 
\end{proof}

This completes the proof of the properties of $\mathcal{Z}$ asserted in the beginning of Theorem \ref{mainthm} and also part (i).

\begin{myproof1}[{Proof of Theorem \ref{mainthm} (ii)}]
We first show that $R\pi_! (\mathcal{Z}_{\mathcal{A}(\mu)}) \cong \IC_\mu$. Since $\overline{\Gr}_{\mathcal{G}, \mu}$ and $\overline{\Gr}_{\underline{G}, \mu}$ are normal, Cohen--Macaulay, and have pseudo-rational singularities then $R\pi_{\mathcal{G},*}(\mathcal{O}_{\overline{\Gr}_{\mathcal{G}, \mu}}) \cong \mathcal{O}_{\overline{\Gr}_{\underline{G}, \mu}}$ by \cite[1.8]{ratsing}. Because $\pi_{\mathcal{G},*} = \pi_{\mathcal{G},!}$ as functors on $\mathbb{F}_p$-sheaves then by applying the Artin--Schreier sequence and Proposition \ref{constprop} it follows that $$R\pi_{\mathcal{G},!}(j_{\mathcal{G}, !*}(\IC_\mu \overset{L}{\boxtimes} \: \restr{\mathbb{F}_p[1]}{C^\circ})) \cong \IC_\mu \overset{L}{\boxtimes} \: \mathbb{F}_p[1].$$  Now by semisimplicity, for general $\mathcal{F}^\bullet \in P_{L^+G}(\Gr, \mathbb{F}_p)^{\text{ss}}$ there exists an isomorphism  
$$R\pi_{\mathcal{G},!}(j_{\mathcal{G}, !*}(\mathcal{F}^\bullet \overset{L}{\boxtimes} \restr{\mathbb{F}_p[1]}{C^\circ})) \cong \mathcal{F}^\bullet \overset{L}{\boxtimes} \mathbb{F}_p[1].$$ We can select a canonical isomorphism by requiring that it restricts to the identity map over $C^\circ$. By restricting this canonical isomorphism to $0$ and applying the proper base change theorem we get a canonical isomorphism 
$$R\pi_!(\mathcal{Z}(\mathcal{F}^\bullet)) \cong \mathcal{F}^\bullet.$$
\end{myproof1}

\section{Proofs} \label{S3}

\subsection{Beilinson--Drinfeld and convolution Grassmannians} \label{BDsec}

In this section we establish some $F$-regularity results that we will use to prove the remaining parts of Theorem \ref{mainthm}. Let $k$ be an algebraically closed field of characteristic $p > 2$ and let $G$ be a connected reductive group over $k$ such that $G_{\text{der}}$ is almost simple. We also assume that $G_{\text{der}}$ is simply connected until Remark \ref{assumption}. In our proofs we will use the same geometric objects as in \cite{ZhuCoherence}. First, we have the Beilinson--Drinfeld Grassmannian for $\mathcal{G}$ over $C$ defined by the functor
$$\Gr_{\mathcal{G}}^{\BD}(R) = \left\{(x, \mathcal{E}, \beta) \: : \: x \in C(R), \: \mathcal{E} \text{ is a } \mathcal{G}\text{-torsor on } C_R, \:  \beta \colon \restr{ \mathcal{E}}{C^{\circ}_R - \Gamma_x} \cong \restr{ \mathcal{E}_0}{C^{\circ}_R - \Gamma_x}    \right\}.$$ This is an ind-proper scheme over $C$ by \cite[6.2.1]{ZhuCoherence}. By arguments similar to those in \cite[Prop. 5]{GaitsgoryCentral},
$$\restr{\Gr_{\mathcal{G}}^{\BD}}{C^\circ} \cong \Gr \times C^\circ \times \mathcal{F}\ell, \quad \quad \quad \quad (\Gr_{\mathcal{G}}^{\BD})_{0} \cong \mathcal{F}\ell.$$ Let $\overline{\Gr}_{\mathcal{G}, \mu, w}^{\BD}$ be the reduced closure of $\Gr_{\leq \mu} \times C^\circ \times \mathcal{F}\ell_{w}$ in $\Gr_{\mathcal{G}}^{\BD}$.

The global convolution Grassmannian for $\mathcal{G}$ is the functor
$$\Gr_{\mathcal{G}}^{\conv}(R) = \left\{ (x, \mathcal{E}_1, \mathcal{E}_2, \beta_1, \beta_2) \: : \:     \begin{array}{lr}
       x \in C(R), \: \mathcal{E}_1, \mathcal{E}_2  \text{ are } \mathcal{G}\text{-torsors on } C_R,\\
       \beta_1 \colon \restr{\mathcal{E}_1}{C_R - \Gamma_x} \cong \restr{ \mathcal{E}_0}{C_R - \Gamma_x}, \: \beta_2 \colon \restr{\mathcal{E}_2}{C_R^\circ} \cong \restr{\mathcal{E}_1}{C_R^\circ} \\
     \end{array} \right\}. 
     $$
This is an ind-projective scheme over $C$ by \cite[6.2.3]{ZhuCoherence}. There is a map $m_{\mathcal{G}} \colon \Gr_{\mathcal{G}}^{\conv} \rightarrow \Gr_{\mathcal{G}}^{\BD}$ which sends $(x, \mathcal{E}_1, \mathcal{E}_2, \beta_1, \beta_2)$ to $(x, \mathcal{E}_2, \beta_1 \circ \beta_2)$. The map $m_{\mathcal{G}}$ is an isomorphism over $C^\circ$. By taking the fibers over $C^\circ$ and $0$ we get the following diagram with Cartesian squares.

\begin{equation} \label{diagram2} \xymatrix{
\restr{\Gr_{\mathcal{G}}^{\conv}}{C^{\circ}} \ar[r]^(.55){j^{\conv}} \ar[d]^{\sim} & \Gr_{\mathcal{G}}^{\conv} \ar[d]^{m_{\mathcal{G}}} & LG \times^I \mathcal{F}\ell \ar[l]_{i^{\conv}} \ar[d]^{m} \\
\restr{\Gr_{\mathcal{G}}^{\BD}}{C^\circ} \ar[r]^(.55){j^{\BD}}   & \Gr_{\mathcal{G}}^{\BD} & \mathcal{F}\ell \ar[l]_(.40){i^{\BD}}  }
\end{equation}

We define $\overline{\Gr}_{\mathcal{G}, \mu, w}^{\conv}$ to be the reduced closure of $\Gr_{\leq \mu} \times C^\circ \times \mathcal{F}\ell_{w}$ in $\Gr_{\mathcal{G}}^{\conv}$. In Theorem \ref{BDF} we will show that $\overline{\Gr}_{\mathcal{G}, \mu, w}^{\BD}$ and $\overline{\Gr}_{\mathcal{G}, \mu, w}^{\conv}$  are strongly $F$-regular. Before we can prove this, we need to show that $\overline{\Gr}_{\mathcal{G}, \mu, w}^{\BD}$ is Frobenius split compatibly with $(\overline{\Gr}_{\mathcal{G}, \mu, w}^{\BD})_{0}$. Zhu proved such a splitting exists for $w \in X_*(T)^+$ sufficiently dominant  \cite[6.5]{ZhuCoherence}, and our argument will require only a minor modification of Zhu's argument.

In our proofs we will use the following facts and notation. Recall that for $\lambda \in X_*(T)^+$ we write $t_\lambda$ for the element $\lambda$ viewed as an element of $\tilde{W}$. If we write $\tilde{W} = X_*(T) \rtimes W$, then $$\pi^{-1}(\Gr_{\leq \lambda}) = \mathcal{F}\ell_{t_{\lambda}^{w_0}}, \quad  \text{where}  \quad t_{\lambda}^{w_0} := (\lambda, w_0) \in \tilde{W}.$$ Additionally, for $\mu \in X_*(T)^+$ we have $$t_\mu \cdot t_{\lambda}^{w_0} = (\mu, 1) \cdot (\lambda, w_0) = (\mu+\lambda, w_0) = t_{\mu+\lambda}^{w_0}.$$ In this case, if $\ell$ is the length function on $\tilde{W}$ we have
$$\ell(t_\mu) + \ell(t_\lambda^{w_0}) = \ell(t_{\mu+\lambda}^{w_0}).$$ This can be proved using the formula for $\ell$ in \cite[\S 1.1]{AffineFlagManifolds} (see also \cite[\S 9.1]{ZhuCoherence}). 

\begin{prop} \label{BDfiber}
Let $\mu$, $\lambda \in X_*(T)^+$ and let $w = t_{\lambda}^{w_0} \in \tilde{W}$. Then $\overline{\Gr}_{\mathcal{G}, \mu, w}^{\BD}$ is normal, and the fiber $(\overline{\Gr}_{\mathcal{G}, \mu, w}^{\BD})_{0}$ is reduced and isomorphic to $\mathcal{F}\ell_{t_{\mu +\lambda}^{w_0}}$.
\end{prop}

\begin{proof}
Let $\Gr_{\underline{G}}^{\BD}$ be the functor 
$$\Gr_{\underline{G}}^{\BD}(R)= \left\{ (x,\mathcal{E}, \beta) \: : \: x \in C(R), \: \mathcal{E} \text{ is a } \underline{G}\text{-torsor on } C_R, \: \beta \colon \restr{\mathcal{E}}{C^\circ_R - \Gamma_x } \cong \restr{\mathcal{E}_0}{C^\circ_R - \Gamma_x} \right\}.$$ Then by \cite[3.1.1]{GaitsgoryCentral}, $\Gr_{\underline{G}}^{\BD}$ is ind-projective, and there are isomorphisms
$$\restr{\Gr_{\underline{G}}^{\BD}}{C^\circ} \cong \Gr  \times C^\circ \times \Gr, \quad \quad \quad \quad (\Gr_{\underline{G}}^{\BD})_{0}  \cong \Gr.$$ The map $\mathcal{G} \rightarrow \underline{G}$ induces a map $\pi_{\BD} \colon \Gr_{\mathcal{G}}^{\BD} \rightarrow \Gr_{\underline{G}}^{\BD}$. Over closed points in $C^\circ$ this is the map $\id \times \pi  \colon \Gr \times \mathcal{F}\ell \rightarrow \Gr \times \Gr $ and over $0$ this is $\pi \colon \mathcal{F}\ell \rightarrow \Gr$. Let $\overline{\Gr}^{\BD}_{\underline{G}, \mu, \lambda} \subset \Gr_{\underline{G}}^{\BD}$ be the reduced closure of $\Gr_{\leq \mu}  \times C^\circ \times
\Gr_{\leq \lambda}$.

The fiber $(\overline{\Gr}^{\BD}_{\underline{G}, \mu, \lambda})_{0}$ is reduced and isomorphic to $\Gr_{\leq \mu + \lambda}$ by \cite[1.2.4]{ZhuAffine} (see also \cite[3.1.14]{ZhuGra}). As $\overline{\Gr}_{\mathcal{G}, \mu, w}^{\BD} \subset \pi_{\BD}^{-1}(\overline{\Gr}^{\BD}_{\underline{G}, \mu, \lambda})$ then $(\overline{\Gr}_{\mathcal{G}, \mu, w}^{\BD})_{0, \, \text{red}} \subset \mathcal{F}\ell_{t_{\mu+\lambda}^{w_0}}$. Furthermore, $\overline{\Gr}_{\mathcal{G}, \mu, w}^{\BD} \rightarrow C$ is flat by \cite[9.7]{HartshorneAG}, and hence by \cite[9.6]{HartshorneAG} the irreducible components of $(\overline{\Gr}_{\mathcal{G}, \mu, w}^{\BD})_{0, \, \text{red}}$ all have dimension equal to $\dim \Gr_{\leq \mu}  + \dim \mathcal{F}\ell_{t_\lambda^{w_0}}$. Thus, since $\dim \mathcal{F} \ell_{t_{\mu+\lambda}^{w_0}} = \dim \Gr_{\leq \mu}  + \dim \mathcal{F}\ell_{t_\lambda^{w_0}}$, then $(\overline{\Gr}_{\mathcal{G}, \mu, w}^{\BD})_{0, \, \text{red}} = \mathcal{F}\ell_{t_{\mu+\lambda}^{w_0}}$ and $( \pi_{\BD}^{-1}(\overline{\Gr}^{\BD}_{\underline{G}, \mu, \lambda}))_{\text{red}} = \overline{\Gr}_{\mathcal{G}, \mu, w}^{\BD}$. 

As $\mathcal{F}\ell$ is a $G/B$ fibration over $\Gr$, then $( \pi_{\BD}^{-1}(\overline{\Gr}^{\BD}_{\underline{G}, \mu, \lambda}))_{0}$ is a $G/B$ fibration over $(\overline{\Gr}^{\BD}_{\underline{G}, \mu, \lambda})_{0}$. Since $(\overline{\Gr}_{\underline{G}, \mu, \lambda}^{\BD})_{0}$ is reduced, then so is $( \pi_{\BD}^{-1}(\overline{\Gr}^{\BD}_{\underline{G}, \mu, \lambda}))_{0}$. As $(\overline{\Gr}_{\mathcal{G}, \mu, w}^{\BD})_{0}$ is a closed subscheme of $( \pi_{\BD}^{-1}(\overline{\Gr}^{\BD}_{\underline{G}, \mu, \lambda}))_{0}$, and these two schemes have the same reductions, then  $(\overline{\Gr}_{\mathcal{G}, \mu, w}^{\BD})_{0}$ is reduced. Finally, since $\mathcal{F}\ell_{t_{\mu+\lambda}^{w_0}}$ is normal and $\restr{\overline{\Gr}_{\mathcal{G}, \mu, w}^{\BD}}{C^\circ} \cong \Gr_{\leq \mu} \times C^\circ \times \mathcal{F}\ell_{t_{\lambda}^{w_0}}$ is also normal then $\overline{\Gr}_{\mathcal{G}, \mu, w}^{\BD}$ is normal by Hironaka's lemma \cite[5.12.8]{EGA4II}.
\end{proof}

\begin{rmrk}
The functors $\Gr_{\mathcal{G}}^{\BD}$ and $\Gr_{\underline{G}}^{\BD}$ can also be described as the restrictions to $C \times \{0\}$ of quotients of global loop groups over $C^2$ which are similar to $\mathcal{L}\mathcal{G}$ and $\mathcal{L}^+\mathcal{G}$ from Section \ref{S.global} (see also \cite[\S 3.1]{ZhuGra}). The  referee has pointed out that it is possible to use this fact to show that $\pi_{\BD}$ is a $G/B$ fibration, and in particular it is smooth. Since normality and reducedness are local in the smooth topology, this leads to an alternative proof of Proposition \ref{BDfiber}.
\end{rmrk}

The following proposition is well known, but because we could not find a complete proof in the literature we provide one here. 

\begin{prop} \label{lengthadd}
Let $w$, $w' \in \tilde{W}$ be such that $\ell(w) + \ell(w') = \ell(w \cdot w')$. Then the convolution morphism $m \colon LG \times^I \mathcal{F}\ell \rightarrow \mathcal{F}{\ell}$ maps $C(w) \simtimes C(w')$ isomorphically onto $C(w \cdot w')$.
\end{prop}

\begin{proof}
Let $W_{\text{aff}}$ be the affine Weyl group of $G$ and let $\Omega \subset \tilde{W}$ be the subgroup of length $0$ elements. Then we have a decomposition $\tilde{W} = W_{\text{aff}} \rtimes \Omega$. First suppose that $w$, $w' \in W_{\text{aff}}$. Let $w = s_1 \cdots s_{\ell(w)}$ and $w' = s_1' \cdots s_{\ell(w')}'$ be reduced expressions for $w $ and $w'$ as products of simple reflections. For each $i$ let $\mathcal{P}_i$ be the parahoric group scheme corresponding to $s_i$. Then $L^+(\mathcal{P}_i)/ I \cong \mathcal{F}\ell_{s_i}$, and each $\mathcal{F}\ell_{s_i}$ is abstractly isomorphic to $\mathbb{P}^1$. There is an affine Demazure resolution
$$\pi_{w\cdot w'} \colon \left( \simtimes \mathcal{F}\ell_{s_i} \right) \simtimes \left( \simtimes \mathcal{F}\ell_{s'_j} \right) \rightarrow \mathcal{F}\ell_{w \cdot w'}.$$ 
By \cite[\S 3]{FaltingsLoop}, since the product of the $s_i$ and the $s_j'$ gives a reduced expression for $w \cdot w'$, the morphism $\pi_{w \cdot w'}$ induces an isomorphism
$$\left( \simtimes C(s_i) \right) \simtimes \left( \simtimes C(s'_j) \right) \xrightarrow{\sim} C(w \cdot w').$$ As $\pi_{w \cdot w'} = m \circ (\pi_w \simtimes \pi_{w'})$  then by also applying \emph{loc. cit.} to the factors $\simtimes \mathcal{F}\ell_{s_i}$ and $\simtimes \mathcal{F}\ell_{s'_j}$ we see that $m$ maps $C(w) \simtimes C(w')$ isomorphically onto $C(w \cdot w')$.

For the general case, write $w = \tau w_{a}$ and $w' = w'_{a}\tau'$ for $w_a$, $w'_a \in W_{\text{aff}}$ and $\tau$, $\tau' \in \Omega$. Then $\ell(w) = \ell(w_a)$, $\ell(w') = \ell(w_a')$ and $\ell(w \cdot w') = \ell(w_a \cdot w'_a)$. Because $\Omega$ normalizes $I$ then $\Omega$ acts on both $\mathcal{F}\ell$ and $LG \times^I \mathcal{F}\ell$ by right multiplication. Furthermore, any lift $\dot{\tau}^{-1} \in LG(k)$ induces isomorphisms on both of these ind-schemes by left multiplication, and all of these isomorphisms send $I$-orbits to $I$-orbits. In particular, there is a commutative diagram as follows:
$$\xymatrix{
C(w) \simtimes C(w') \ar[rr]^{\dot{\tau}^{-1} \: ( \cdot ) \: \tau'^{-1}}_{\sim} \ar[d]^m & & C(w_a) \simtimes C(w'_a) \ar[d]^m \\
C(w \cdot w') \ar[rr]^{\dot{\tau}^{-1} \: ( \cdot )\:  \tau'^{-1}}_{\sim}   & & C(w_a \cdot w'_a)   }
$$
We have shown the morphism on the right is an isomorphism, thus so is the morphism on the left.
\end{proof}

\begin{prop}
For any $\mu \in X_*(T)^+$ and $w \in \tilde{W}$ the scheme $\overline{\Gr}_{\mathcal{G}, \mu, w}^{\BD}$ is Frobenius split compatibly with $(\overline{\Gr}_{\mathcal{G}, \mu, w}^{\BD})_{0}$.
\end{prop}

\begin{proof}
If $w = t_\nu$ for $\nu \in X_*(T)^+$ sufficiently dominant then this is \cite[6.5]{ZhuCoherence}. If $\lambda < \nu$ is also dominant then this splitting is compatible with the closed subscheme $\overline{\Gr}_{\mathcal{G}, \mu, t_\lambda}^{\BD} \subset \overline{\Gr}_{\mathcal{G}, \mu, t_\nu}^{\BD}$ by \cite[6.8]{ZhuCoherence} and \cite[1.1.7 (ii)]{BrionKumar}. A splitting of $\overline{\Gr}_{\mathcal{G}, \mu, t_\nu}^{\BD}$ compatible with $\overline{\Gr}_{\mathcal{G}, \mu, t_\lambda}^{\BD}$ and $(\overline{\Gr}_{\mathcal{G}, \mu, t_\nu}^{\BD})_{0}$ induces a splitting of $\overline{\Gr}_{\mathcal{G}, \mu, t_\lambda}^{\BD}$ compatible with $(\overline{\Gr}_{\mathcal{G}, \mu, t_\lambda}^{\BD})_{0}$. This proves the proposition when $w = t_\lambda$ for any dominant cocharacter $\lambda$. 

For the general case, note that for every $w' \in \tilde{W}$ there exists $\lambda \in X_*(T)^+$ such that $w' \leq t_{\lambda}^{w_0}$, so it suffices to show $\overline{\Gr}_{\mathcal{G}, \mu, t_{\lambda}^{w_0}}^{\BD}$ is compatibly Frobenius split with $(\overline{\Gr}_{\mathcal{G}, \mu, t_{\lambda}^{w_0}}^{\BD})_{0}$ and $\overline{\Gr}_{\mathcal{G}, \mu, w'}^{\BD}$ for all $\lambda \in X_*(T)^+$ and $w' \leq t_{\lambda}^{w_0}$. Henceforth we fix $\lambda \in X_*(T)^+$ and $w = t_{\lambda}^{w_0} \in \tilde{W}$. 

We will proceed by a similar argument as in \cite[6.7]{ZhuCoherence}. More precisely, we will construct an open subscheme $U \subset \overline{\Gr}_{\mathcal{G}, \mu, w}^{\conv}$ such that:
\begin{enumerate}
\item[(i)] The scheme $U$ maps isomorphically onto its image under $m \colon  \overline{\Gr}_{\mathcal{G}, \mu, w}^{\conv} \rightarrow \overline{\Gr}^{\BD}_{\mathcal{G}, \mu, w}$.
\item[(ii)] The complement of $m(U)$ in $\overline{\Gr}_{\mathcal{G}, \mu, w}^{\BD}$ has codimension two.
\item[(iii)] The scheme $U$ is Frobenius split. Since $\overline{\Gr}_{\mathcal{G}, \mu, w}^{\BD}$ is normal, then by \cite[1.1.7]{BrionKumar} the spitting of $m(U)$ extends to a splitting of $\overline{\Gr}_{\mathcal{G}, \mu, w}^{\BD}$. We will complete the proof of the proposition by showing
\item[(iv)] The resulting splitting of $\overline{\Gr}_{\mathcal{G}, \mu, w}^{\BD}$ is compatible with $(\overline{\Gr}_{\mathcal{G}, \mu, w}^{\BD})_{0}$ and $\overline{\Gr}_{\mathcal{G}, \mu, w'}^{\BD}$ for all $w' < w$. 
\end{enumerate}

As in \cite{ZhuCoherence}, we define $U_1 \subset \overline{\Gr}_{\mathcal{G}, \mu, w}^{\conv}$ to be the open subscheme which is $\Gr_{\mu} \times C^\circ \times \mathcal{F}\ell_{w}$ over $C^\circ$ and $C(t_\mu) \simtimes \mathcal{F}\ell_{w}$ over $0$. We also define $U \subset U_1$ to be the  open subscheme which is $\Gr_{\mu} \times C^\circ \times \mathcal{F}\ell_{w}$ over $C^\circ$ and $C(t_\mu) \simtimes C(w)$ over $0$. Since the lengths of $t_\mu$ and $w$ add then $C(t_\mu) \simtimes C(w)$ maps isomorphically onto $C(t_{\mu+\lambda}^{w_0})$ by Proposition \ref{lengthadd}. Thus $(U)_0$ maps isomorphically onto $m(U)_0$. Hence the morphism from $U$ to $m(U)$ is a bijective birational morphism between normal integral $k$-schemes, so it is an isomorphism by Grothendieck's reformulation of Zariski's main theorem \cite[4.4.3]{EGA3I}.

As $\Gr_{\leq \mu} - \Gr_{\mu}$ has codimension two in $\Gr_{\leq \mu}$, the complement of $\restr{m(U)}{C^\circ}   = \Gr_{\mu} \times C^\circ \times \mathcal{F}\ell_{w}$ has codimension two in $\restr{\overline{\Gr}_{\mathcal{G}, \mu, w}^{\BD}}{C^\circ} = \Gr_{\leq \mu} \times C^\circ \times \mathcal{F}\ell_{w}$. By Proposition \ref{BDfiber}, $(\overline{\Gr}_{\mathcal{G}, \mu, w}^{\BD})_{0} - m(U)_{0} = \mathcal{F}\ell_{t_{\mu+\lambda}^{w_0}} - C(t_{\mu+\lambda}^{w_0})$ has codimension one, so we conclude that $U$ satisfies (ii). Finally, (iii) and (iv) follow from the fact that $U_1$ is Frobenius split compatibly with $(U_1)_{0}$ and $U_1 \cap (\Gr_{\mu} \times C^\circ \times \mathcal{F}\ell_{w'})$ for all $w' \leq w$ by \cite[6.8]{ZhuCoherence}. 
\end{proof}

\begin{cor}
For any $\mu \in X_*(T)^+$ and $w \in \tilde{W}$ the scheme $(\overline{\Gr}_{\mathcal{G}, \mu, w}^{\BD})_{0}$ is reduced, Cohen--Macaulay, Frobenius split, and equidimensional of dimension $\dim \Gr_{\leq \mu} + \dim \mathcal{F}\ell_w$. 
\end{cor}

\begin{proof}
A Frobenius splitting of $\overline{\Gr}_{\mathcal{G}, \mu, w}^{\BD}$ compatible with $(\overline{\Gr}_{\mathcal{G}, \mu, w}^{\BD})_{0}$ induces a Frobenius splitting of $(\overline{\Gr}_{\mathcal{G}, \mu, w}^{\BD})_{0}$. Thus $(\overline{\Gr}_{\mathcal{G}, \mu, w}^{\BD})_{0}$ is reduced by \cite[1.2.1]{BrionKumar}. Furthermore, as $\restr{\overline{\Gr}_{\mathcal{G}, \mu, w}^{\BD}}{C^\circ} \cong \Gr_{\leq \mu} \times C^\circ \times \mathcal{F}\ell_{w}$ is Cohen--Macaulay then $\overline{\Gr}_{\mathcal{G}, \mu, w}^{\BD}$ is Cohen--Macaulay by \cite[5.4]{BlickleSchwede} (see also \cite[5.5]{RicharzLocalModels}). Thus $(\overline{\Gr}_{\mathcal{G}, \mu, w}^{\BD})_{0}$ is also Cohen--Macaulay by \cite[OC6G]{stacks-project}. Finally, the morphism $\overline{\Gr}_{\mathcal{G}, \mu, w}^{\BD} \rightarrow C$ is flat by \cite[III 9.7]{HartshorneAG}, so by \cite[III 9.6]{HartshorneAG} the fiber $(\overline{\Gr}_{\mathcal{G}, \mu, w}^{\BD})_{0}$ is equidimensional of dimension $\dim \Gr_{\leq \mu} + \dim \mathcal{F}\ell_w$.
\end{proof}

\begin{thm} \label{BDF}
For any $\mu \in X_*(T)^+$ and $w \in \tilde{W}$ the schemes $\overline{\Gr}_{\mathcal{G}, \mu, w}^{\BD}$ and $\overline{\Gr}_{\mathcal{G}, \mu, w}^{\conv}$  are strongly $F$-regular, $F$-rational, and have pseudo-rational singularities. 
\end{thm}

\begin{proof}
As in the proof of Theorem \ref{BDprops} it suffices to prove these schemes are strongly $F$-regular. Since $\restr{\overline{\Gr}_{\mathcal{G}, \mu, w}^{\BD}}{C^\circ} \cong \Gr_{\leq \mu} \times C^\circ \times \mathcal{F}\ell_{w}$ then $\restr{\overline{\Gr}_{\mathcal{G}, \mu, w}^{\BD}}{C^\circ}$ is strongly $F$-regular by Lemma \ref{prodlem} and Theorem \ref{globFreg}. Now as in the proof of Theorem \ref{BDprops} it follows that $\overline{\Gr}_{\mathcal{G}, \mu, w}^{\BD}$ is strongly $F$-regular because it is Frobenius split compatibly with $(\overline{\Gr}_{\mathcal{G}, \mu, w}^{\BD})_{0}$. 

To prove $\overline{\Gr}_{\mathcal{G}, \mu, w}^{\conv}$ is strongly $F$-regular, we first note that $\overline{\Gr}_{\mathcal{G}, \mu} \times \mathcal{F}\ell_{w}$ is strongly $F$-regular by Lemma \ref{prodlem}. By the isomorphism proceeding \cite[6.2.3]{ZhuCoherence}, $\overline{\Gr}_{\mathcal{G}, \mu, w}^{\conv}$ and $\overline{\Gr}_{\mathcal{G}, \mu} \times \mathcal{F}\ell_{w}$ have a common smooth cover. Since the property of strong $F$-regularity is local in the smooth topology (Lemma \ref{Fregdescent}) then $\overline{\Gr}_{\mathcal{G}, \mu, w}^{\conv}$ is strongly $F$-regular. 
\end{proof}

\subsection{Proofs of main results}

We continue using the notation introduced in Section \ref{BDsec}. For $\mu \in X_*(T)^+$ and $w \in \tilde{W}$ let $\mathcal{Z}_{\mu, w}$ be the shifted constant sheaf $\mathbb{F}_p[\dim (\Gr_{\mathcal{G}, \mu, w}^{\BD})_{0}]$ supported on $(\Gr_{\mathcal{G}, \mu, w}^{\BD})_{0}$. Parts (iii) and (iv) of Theorem \ref{mainthm} follow from the following proposition.

\begin{prop}
For $\mathcal{F}_1^\bullet \in P_{L^+G}(\Gr, \mathbb{F}_p)^{\sss}$ and $\mathcal{F}_2^{\bullet} \in P_I(\mathcal{F}\ell, \mathbb{F}_p)$ there are natural isomorphisms
\leavevmode
\begin{enumerate}[{\normalfont (i)}]
\item $\Psi_{\Gr^{\BD}_{\mathcal{G}}}(\IC_\mu \overset{L}{\boxtimes} \, \restr{\mathbb{F}_p[1]}{C^\circ} \overset{L}{\boxtimes} \IC_w^{\mathcal{F}\ell}) \cong \mathcal{Z}_{\mu, w}.$
\item $\Psi_{\Gr^{\BD}_{\mathcal{G}}}(\mathcal{F}_1^\bullet \overset{L}{\boxtimes} \restr{\mathbb{F}_p[1]}{C^\circ} \overset{L}{\boxtimes} \mathcal{F}_2^{\bullet}) \cong \mathcal{Z}(\mathcal{F}_1^\bullet) * \mathcal{F}_2^{\bullet}.$
\item $\Psi_{\Gr^{\BD}_{\mathcal{G}}}(\mathcal{F}_1^\bullet \overset{L}{\boxtimes} \restr{\mathbb{F}_p[1]}{C^\circ} \overset{L}{\boxtimes} \mathcal{F}_2^{\bullet}) \cong \mathcal{F}_2^{\bullet} * \mathcal{Z}(\mathcal{F}_1^{\bullet}).$
\end{enumerate}
Furthermore, each of these complexes is perverse and $I$-equivariant.
\end{prop}

\begin{proof}
Since $\overline{\Gr}_{\mathcal{G}, \mu, w}^{\BD}$ is integral and $F$-rational then by \cite[1.7]{modpGr}, $j_{!*}^{\BD}(\IC_\mu \overset{L}{\boxtimes} \, \restr{\mathbb{F}_p[1]}{C^\circ} \overset{L}{\boxtimes} \IC_w^{\mathcal{F}\ell} )$ is the constant sheaf $\mathbb{F}_p[\dim \overline{\Gr}_{\mathcal{G}, \mu, w}^{\BD}]$ supported on $\overline{\Gr}_{\mathcal{G}, \mu, w}^{\BD}$. Hence the isomorphism in (i) follows. As $(\overline{\Gr}_{\mathcal{G}, \mu, w}^{\BD})_{0}$ is Cohen--Macaulay and equidimensional  then $\mathcal{Z}_{\mu, w}$ is perverse by \cite[1.6]{modpGr}.

For (ii), suppose that $\mathcal{F}_1^\bullet$ is supported on $\Gr_{\leq \mu}$ and $\mathcal{F}_2^{\bullet}$ is supported on $\mathcal{F}\ell_{w}$. Let $I_n := L^n(\restr{\mathcal{G}}{{\hat{\mathcal{O}}_{0}}})$. As in \cite[6.2.3]{ZhuCoherence}, for some $n$ there is an $I_n$-torsor $\Gr_{\mathcal{G}, 0 ,n}$ over $\Gr_{\mathcal{G}}$ such that 
$$\overline{\Gr}_{\mathcal{G}, \mu, w}^{\conv} \cong \overline{\Gr}_{\mathcal{G}, \mu} \times_{\Gr_{\mathcal{G}}}  \Gr_{\mathcal{G}, 0 ,n} \times^{I_n} \mathcal{F}\ell_{w}.$$ Let $\phi_n^{\conv} \colon \overline{\Gr}_{\mathcal{G}, \mu} \times_{\Gr_{\mathcal{G}}}  \Gr_{\mathcal{G}, 0 ,n} \times \mathcal{F}\ell_{w} \rightarrow \overline{\Gr}_{\mathcal{G}, \mu, w}^{\conv}$ be the resulting $I_n$-torsor over $\overline{\Gr}_{\mathcal{G}, \mu, w}^{\conv}$.

By similar reasoning as in \cite[6.2]{modpGr} we can form the perverse sheaf
$$j_{\mathcal{G}, !*}(\mathcal{F}_1^\bullet \overset{L}{\boxtimes} \restr{\mathbb{F}_p[1]}{C^\circ}) \overset{\sim}{\boxtimes} \mathcal{F}_2^{\bullet} \in P_{\mathcal{L}^+ \mathcal{G}}(\overline{\Gr}_{\mathcal{G}, \mu, w}^{\conv}, \mathbb{F}_p).$$ The key point is that $j_{\mathcal{G}, !*}(\mathcal{F}_1^\bullet \overset{L}{\boxtimes} \restr{\mathbb{F}_p[1]}{C^\circ}) {\overset{L}{\boxtimes}} \, \mathcal{F}_2^{\bullet}$ is perverse when $\mathcal{F}_1^\bullet$ and $\mathcal{F}_2^{\bullet}$ are simple because then $j_{\mathcal{G}, !*}(\mathcal{F}_1^\bullet \overset{L}{\boxtimes} \restr{\mathbb{F}_p[1]}{C^\circ}) \overset{L}{\boxtimes} \mathcal{F}_2^{\bullet}$ is a constant sheaf supported on an equidimensional Cohen--Macaulay scheme. Thus $j_{\mathcal{G}, !*}(\mathcal{F}_1^\bullet \overset{L}{\boxtimes} \restr{\mathbb{F}_p[1]}{C^\circ}) {\overset{L}{\boxtimes}} \, \mathcal{F}_2^{\bullet}$ is perverse for general $\mathcal{F}_1^\bullet$ and $\mathcal{F}_2^{\bullet}$ by induction on their lengths.

 We claim there is an isomorphism 
\begin{equation} \label{eq4} Ri^{\conv,*}[-1](j_{\mathcal{G}, !*}(\mathcal{F}_1^\bullet \overset{L}{\boxtimes} \restr{\mathbb{F}_p[1]}{C^\circ}) \overset{\sim}{\boxtimes} \mathcal{F}_2^{\bullet}) \cong \mathcal{Z}(\mathcal{F}_1^\bullet) \overset{\sim}{\boxtimes} \mathcal{F}_2^{\bullet}.
\end{equation} 

Let $\phi_n \colon  \overline{\Gr}_{\mathcal{G}, \mu} \times_{\Gr_{\mathcal{G}}}  \Gr_{\mathcal{G}, 0 ,n} \rightarrow \overline{\Gr}_{\mathcal{G}, \mu}$ be the pullback of $\Gr_{\mathcal{G}, 0 ,n} \rightarrow \Gr_{\mathcal{G}}$ along $\overline{\Gr}_{\mathcal{G}, \mu} \rightarrow \Gr_{\mathcal{G}}$ and let $\phi_{0, n} \colon (\overline{\Gr}_{\mathcal{G},\mu})_{0, n} \rightarrow (\overline{\Gr}_{\mathcal{G}, \mu})_{0} $ be the fiber of $\phi_n$ over $0$.  By taking the fiber of $\phi_n^{\conv}$ over $0$ we get the following Cartesian diagram:
$$
\xymatrix{
\overline{\Gr}_{\mathcal{G}, \mu} \times_{\Gr_{\mathcal{G}}}  \Gr_{\mathcal{G}, 0 ,n} \times \mathcal{F}\ell_{w} \ar[d]^{\phi_n^{\conv}} & (\overline{\Gr}_{\mathcal{G},\mu})_{0,n} \times \mathcal{F}\ell_{w}  \ar[l]_(.4){i_{n}^{\conv}} \ar[d]^{\phi_{0, n}^{\conv}}\\
  \overline{\Gr}_{\mathcal{G}, \mu, w}^{\conv} & \ar[l]_{i^{\conv}} (\overline{\Gr}_{\mathcal{G},\mu})_{0} \simtimes \mathcal{F}\ell_{w}}
$$
Then (\ref{eq4}) follows from the following calculation:
\begin{align*} R\phi_{0,n}^{\conv, *} &(Ri^{\conv,*}[-1] (j_{\mathcal{G}, !*}(\mathcal{F}_1^\bullet \overset{L}{\boxtimes} \restr{\mathbb{F}_p[1]}{C^\circ}) \overset{\sim}{\boxtimes} \mathcal{F}_2^{\bullet})) \\ & \cong  Ri_n^{\conv,*}[-1] (R\phi_n^*(j_{\mathcal{G}, !*}(\mathcal{F}_1^\bullet \overset{L}{\boxtimes} \restr{\mathbb{F}_p[1]}{C^\circ})) \overset{L}{\boxtimes} \mathcal{F}_2^{\bullet} ) \\
& \cong  R\phi_{0, n}^*(Ri_{\mathcal{G}}^*[-1](j_{\mathcal{G}, !*}(\mathcal{F}_1^\bullet \overset{L}{\boxtimes} \restr{\mathbb{F}_p[1]}{C^\circ})))  \overset{L}{\boxtimes} \mathcal{F}_2^{\bullet}  \cong R\phi_{0, n}^*(\mathcal{Z}(\mathcal{F}_1^\bullet)) \overset{L}{\boxtimes} \mathcal{F}_2^{\bullet}.
\end{align*} 

To finish the proof of (ii) it suffices to construct a natural isomorphism
\begin{equation} \label{eq5}
 Rm_*( Ri^{\conv,*}[-1](j_{\mathcal{G}, !*}(\mathcal{F}_1^\bullet \overset{L}{\boxtimes} \restr{\mathbb{F}_p[1]}{C^\circ}) \: {\overset{\sim}{\boxtimes}} \: \mathcal{F}_2^{\bullet})) \rightarrow \Psi_{\Gr^{\BD}_{\mathcal{G}}}(\mathcal{F}_1^\bullet \overset{L}{\boxtimes} \restr{\mathbb{F}_p[1]}{C^\circ} \overset{L}{\boxtimes} \mathcal{F}_2^{\bullet}).
\end{equation} By the proper base change theorem, to establish an isomorphism as in (\ref{eq5}) it suffices to construct an isomorphism 
\begin{equation} \label{eq6} 
Rm_{\mathcal{G},*}(j_{\mathcal{G}, !*}(\mathcal{F}_1^\bullet \overset{L}{\boxtimes} \restr{\mathbb{F}_p[1]}{C^\circ}) \: {\overset{\sim}{\boxtimes}} \: \mathcal{F}_2^{\bullet}) \rightarrow j^{\BD}_{!*}(\mathcal{F}_1^\bullet \overset{L}{\boxtimes} \restr{\mathbb{F}_p[1]}{C^\circ} \overset{L}{\boxtimes} \mathcal{F}_2^{\bullet}).
\end{equation} Since $m_{\mathcal{G}}$ is an isomorphism over $C^\circ$ then the left side of (\ref{eq6}) is naturally an extension of $\mathcal{F}_1^\bullet \overset{L}{\boxtimes} \restr{\mathbb{F}_p[1]}{C^\circ} \overset{L}{\boxtimes} \mathcal{F}_2^{\bullet}$, so we just have to show it is the intermediate extension. Indeed, once we know the left side is isomorphic to the intermediate extension, then by \cite[2.11]{modpGr} there is a unique isomorphism as in (\ref{eq6}) which restricts to the identity map on $\mathcal{F}_1^\bullet \overset{L}{\boxtimes} \restr{\mathbb{F}_p[1]}{C^\circ} \overset{L}{\boxtimes} \mathcal{F}_2^{\bullet}$.

Since the middle term in an exact triangle is an intermediate extension if the outer terms are intermediate extensions (see the proof of \cite[7.8]{modpGr}) then by induction on the lengths of $\mathcal{F}_1^\bullet$ and $\mathcal{F}_2^{\bullet}$ we reduce to the case $\mathcal{F}_1^\bullet = \IC_{\mu}$ and $\mathcal{F}_2^{\bullet} = \IC_w^{\mathcal{F}\ell}$. Because all of the schemes appearing are integral and $F$-rational, then both  $j_{\mathcal{G}, !*}(\IC_\mu \overset{L}{\boxtimes} \, \restr{\mathbb{F}_p[1]}{C^\circ}) \overset{\sim}{\boxtimes} \IC_{w}^{\mathcal{F}\ell}$ and $j_{!*}^{\BD}(\IC_\mu \overset{L}{\boxtimes} \, \restr{\mathbb{F}_p[1]}{C^\circ} \overset{L}{\boxtimes} \IC_{w}^{\mathcal{F}\ell})$ are constant sheaves.

The map $m \colon \overline{\Gr}_{\mathcal{G}, \mu, w}^{\conv} \rightarrow \overline{\Gr}_{\mathcal{G}, \mu, w}^{\BD}$ is a birational map between normal, Cohen--Macaulay $k$-schemes having pseudo-rational singularities. Thus by \cite[1.8]{ratsing}, $Rm_*(\mathcal{O}_{\overline{\Gr}_{\mathcal{G}, \mu, w}^{\conv}}) \cong \mathcal{O}_{\overline{\Gr}_{\mathcal{G}, \mu, w}^{\BD}}$. By applying the Artin--Schreier sequence we complete the proof of (\ref{eq6}) and (ii). Applying (i), this also shows that $\mathcal{Z}(\mathcal{F}_1^\bullet) * \mathcal{F}_2^{\bullet}$ is perverse if $\mathcal{F}_1^\bullet = \IC_{\mu}$ and $\mathcal{F}_2^{\bullet} = \IC_w^{\mathcal{F}\ell}$. Convolution on the left by $\mathcal{Z}(\mathcal{F}_1^\bullet)$ sends short exact sequences in $P_I(\mathcal{F}\ell, \mathbb{F}_p)$ to exact triangles in $D_c^b(\mathcal{F}\ell_{\text{\'{e}t}}, \mathbb{F}_p)$ by a proof analogous to the one in \cite[6.7]{modpGr}. Thus, by induction on the lengths of $\mathcal{F}_1^\bullet$ and $\mathcal{F}_2^{\bullet}$ we conclude that $\mathcal{Z}(\mathcal{F}_1^\bullet) * \mathcal{F}_2^{\bullet}$ is perverse in general. For equivariance, we note that because $\mathcal{Z}(\mathcal{F}_1^\bullet)$ is $I$-equivariant then $\mathcal{Z}(\mathcal{F}_1^\bullet) \overset{\sim}{\boxtimes} \mathcal{F}_2^{\bullet}$ is $I$-equivariant for the action of $I$ on the left factor of $LG \times^I \mathcal{F}\ell$. As the map $m$ is $I$-equivariant then $\mathcal{Z}(\mathcal{F}_1^\bullet) * \mathcal{F}_2^{\bullet}$ $I$-equivariant by \cite[3.2]{modpGr}.

To prove (iii) we use the functor 
$$\Gr_{\mathcal{G}}^{\conv'}(R) = \left\{ (x, \mathcal{E}_1, \mathcal{E}_2, \beta_1, \beta_2) \: : \:     \begin{array}{lr}
       x \in C(R), \: \mathcal{E}_1, \mathcal{E}_2  \text{ are } \mathcal{G}\text{-torsors on } C_R,\\
       \beta_1 \colon \restr{\mathcal{E}_1}{C_R^\circ} \cong \restr{ \mathcal{E}_0}{C_R^\circ}, \: \beta_2 \colon \restr{\mathcal{E}_2}{C_R - \Gamma_x} \cong \restr{\mathcal{E}_1}{C_R - \Gamma_x} \\
     \end{array} \right\}. 
     $$
By \cite[7.2.6]{ZhuCoherence}, $\Gr_{\mathcal{G}}^{\conv'}$ is ind-proper over $C$. There is also a map $m_{\mathcal{G}}' \colon  \Gr_{\mathcal{G}}^{\conv'} \rightarrow \Gr_{\mathcal{G}}^{\BD}$ which sends $(x, \mathcal{E}_1, \mathcal{E}_2, \beta_1, \beta_2)$ to $(x, \mathcal{E}_2, \beta_1 \circ \beta_2)$. The map $m_{\mathcal{G}}'$ is an isomorphism over $C^\circ$ and it restricts to the convolution map $m \colon LG \times^I \mathcal{F}\ell \rightarrow \mathcal{F}\ell$ over $0$. 

Suppose $\mathcal{F}_1^\bullet$ is supported on $\Gr_{\leq \mu}$ and $\mathcal{F}_2^{\bullet}$ is supported on $\mathcal{F}\ell_{w}$. Let $n$ be an integer large enough so that $\mathcal{L}^+ \mathcal{G}$ acts on $\overline{\Gr}_{\mathcal{G}, \mu}$ through the quotient $\mathcal{L}^+_n \mathcal{G}$. Then as in the proof of \cite[7.4 (ii)]{ZhuCoherence} there is an $\mathcal{L}^+_n \mathcal{G}$-torsor $\mathcal{P}_n$ over $\mathcal{F}\ell_{w} \times C$ such that $\mathcal{P}_n \times^{\mathcal{L}^+_n \mathcal{G}} \overline{\Gr}_{\mathcal{G}, \mu} \subset \overline{\Gr}_{\mathcal{G}}^{\conv'}$ is a closed subscheme with
$$\restr{\mathcal{P}_n \times^{\mathcal{L}^+_n \mathcal{G}} \overline{\Gr}_{\mathcal{G}, \mu}}{C^\circ} \cong \mathcal{F}\ell_{w} \times C^\circ \times \Gr_{\leq \mu}, \quad (\mathcal{P}_n \times^{\mathcal{L}^+_n \mathcal{G}} \overline{\Gr}_{\mathcal{G}, \mu} )_{0} \cong \mathcal{F}\ell_{w} \simtimes (\overline{\Gr}_{\mathcal{G}, \mu})_{0}.$$ The scheme $\mathcal{P}_n \times^{L^+_n \mathcal{G}} \overline{\Gr}_{\mathcal{G}, \mu}$ is strongly $F$-regular by Lemma \ref{prodlem} and because this property is local in the smooth topology by Lemma \ref{Fregdescent}. Let $\phi_n^{\conv'} \colon \mathcal{P}_n \times_C \overline{\Gr}_{\mathcal{G}, \mu} \rightarrow \mathcal{P}_n \times^{\mathcal{L}^+_n \mathcal{G}} \overline{\Gr}_{\mathcal{G}, \mu}$ be the resulting $\mathcal{L}^+_n \mathcal{G}$-torsor.

Since $j_{\mathcal{G}, !*}(\mathcal{F}_1^\bullet \overset{L}{\boxtimes} \restr{\mathbb{F}_p[1]}{C^\circ})$ is $\mathcal{L}^+_n \mathcal{G}$-equivariant (see the proof of Proposition \ref{equivprop}), then by arguments analogous to those in the proof of (ii) we can form the perverse sheaf
$$(\mathcal{F}_2^{\bullet} \overset{L}{\boxtimes} \mathbb{F}_p[1]) \overset{\sim}{\boxtimes} j_{\mathcal{G}, !*}(\mathcal{F}_1^\bullet \overset{L}{\boxtimes} \restr{\mathbb{F}_p[1]}{C^\circ}) \in P_c^b(\overline{\Gr}_{\mathcal{G}}^{\conv'}, \mathbb{F}_p),$$ which is supported on $\mathcal{P}_n \times^{L^+_n \mathcal{G}} \overline{\Gr}_{\mathcal{G}, \mu}$. Here we are applying the operation $\overset{\sim}{\boxtimes}$ with respect to the fiber product of $C$-schemes rather than $k$-schemes. If $\mathcal{F}_1^\bullet = \IC_\mu$ and $\mathcal{F}_2^{\bullet} = \IC_{w}^{\mathcal{F}\ell}$ then $(\IC_{w}^{\mathcal{F}\ell} \overset{L}{\boxtimes} \, \mathbb{F}_p[1]) \overset{\sim}{\boxtimes} j_{\mathcal{G}, !*}(\IC_\mu \overset{L}{\boxtimes} \, \restr{\mathbb{F}_p[1]}{C^\circ})$ is a constant sheaf. Let $\mathcal{F}\ell_{w, n}$ be the $I_n$-torsor $(\mathcal{P}_n)_{0}$ over $\mathcal{F}\ell_{w}$. By taking the fiber over $0$ we get the following Cartesian diagram:
$$
\xymatrix{
\mathcal{P}_n \times_C \overline{\Gr}_{\mathcal{G}, \mu}  \ar[d]^{\phi_n^{\conv'}} & \mathcal{F}\ell_{w, n} \times (\overline{\Gr}_{\mathcal{G}, \mu})_{0}  \ar[l]_{i_{n}^{\conv'}} \ar[d]^{\phi_{0,n}^{\conv'}}\\
  \mathcal{P}_n \times^{\mathcal{L}^+_n \mathcal{G}} \overline{\Gr}_{\mathcal{G}, \mu}  & \ar[l]_{i^{\conv'}} \mathcal{F}\ell_{w} \simtimes (\overline{\Gr}_{\mathcal{G}, \mu})_{0}}
$$

Now we can finish the proof by following arguments analogous to those in the proof of (ii) to establish isomorphisms
$$Ri^{\conv', *}[-1]((\mathcal{F}_2^{\bullet} \overset{L}{\boxtimes} \mathbb{F}_p[1]) \overset{\sim}{\boxtimes} j_{\mathcal{G}, !*}(\mathcal{F}_1^\bullet \overset{L}{\boxtimes} \restr{\mathbb{F}_p[1]}{C^\circ})) \cong \mathcal{F}_2^{\bullet} \overset{\sim}{\boxtimes} \mathcal{Z}(\mathcal{F}_1^\bullet),$$ and
$$
Rm_{*}(Ri^{\conv', *}[-1]((\mathcal{F}_2^{\bullet} \overset{L}{\boxtimes} \mathbb{F}_p[1]) \overset{\sim}{\boxtimes} j_{\mathcal{G}, !*}(\mathcal{F}_1^\bullet \overset{L}{\boxtimes} \restr{\mathbb{F}_p[1]}{C^\circ}))) \cong \Psi_{\Gr^{\BD}_{\mathcal{G}}}(\mathcal{F}_1^\bullet \overset{L}{\boxtimes} \restr{\mathbb{F}_p[1]}{C^\circ} \overset{L}{\boxtimes} \mathcal{F}_2^{\bullet}).
$$
In this last isomorphism we are applying the natural isomorphism $\mathcal{F}\ell_w \times C^\circ \times \Gr_{\leq \mu} \cong \Gr_{\leq \mu}  \times C^\circ \times \mathcal{F}\ell_w$. We leave the details to the reader.
\end{proof}

Before we finish the proof of Theorem \ref{mainthm} we introduce the functor
$$\Gr_{\mathcal{G}}^{\conv''}(R) = \left\{  (x, \mathcal{E}_1, \mathcal{E}_2, \beta_1, \beta_2) \: : \:    \begin{array}{lr}
       x \in C(R), \: \mathcal{E}_1, \mathcal{E}_2  \text{ are } \mathcal{G}\text{-torsors on } C_R\\
       \beta_1 \colon \restr{\mathcal{E}_1}{C_R - \Gamma_x} \cong \restr{ \mathcal{E}_0}{C_R - \Gamma_x}, \: \beta_2 \colon \restr{\mathcal{E}_2}{C_R-\Gamma_x} \cong \restr{\mathcal{E}_1}{C_R-\Gamma_x} \\
     \end{array} \right\}.
     $$
There are isomorphisms
$$
\restr{\Gr_{\mathcal{G}}^{\conv''}}{C^\circ} \cong (LG \times^{L^+G} \Gr) \times C^\circ, \quad \quad \quad \quad (\Gr_{\mathcal{G}}^{\conv''})_{0} \cong  LG \times^I \mathcal{F}\ell.
$$ Moreover, by arguments similar to those proceeding \cite[6.2.3]{ZhuCoherence}, 
$$\Gr_{\mathcal{G}}^{\conv''} \cong \mathcal{L}\mathcal{G} \times^{\mathcal{L}^+ \mathcal{G}} \Gr_{\mathcal{G}} =: \Gr_{\mathcal{G}} \simtimes \Gr_{\mathcal{G}}.$$ Thus $\Gr_{\mathcal{G}}^{\conv''}$ is ind-proper over $C$. We define $\overline{\Gr}^{\conv''}_{\mathcal{G}, \mu, \lambda}$ to be the reduced closure of the subscheme $\Gr_{\leq \mu} \simtimes \Gr_{\leq \lambda} \times C^{\circ} \subset \Gr_{\mathcal{G}}^{\conv''}$.

\begin{prop} \label{CMProp}
The scheme $\overline{\Gr}^{\conv''}_{\mathcal{G}, \mu, \lambda}$ is Cohen--Macaulay, $(\overline{\Gr}^{\conv''}_{\mathcal{G}, \mu, \lambda})_0$ is reduced, and 
$$\overline{\Gr}^{\conv''}_{\mathcal{G}, \mu, \lambda} \cong (\overline{\Gr}_{\mathcal{G},\mu} \simtimes \overline{\Gr}_{\mathcal{G},\lambda})_{\textnormal{red}}.$$
\end{prop}

\begin{proof}
Note that $\overline{\Gr}^{\conv''}_{\mathcal{G}, \mu, \lambda}$ is a closed subscheme of  $\overline{\Gr}_{\mathcal{G},\mu} \simtimes \overline{\Gr}_{\mathcal{G},\lambda}$, and these two schemes are isomorphic over $C^\circ$. Thus the isomorphism in the proposition will follow if we show that $\overline{\Gr}_{\mathcal{G},\mu} \simtimes \overline{\Gr}_{\mathcal{G},\lambda}$ is irreducible.

To prove that $\overline{\Gr}_{\mathcal{G},\mu} \simtimes \overline{\Gr}_{\mathcal{G},\lambda}$ is irreducible, fix an integer $n$ large enough so that $\mathcal{L}^+\mathcal{G}$ acts on $\overline{\Gr}_{\mathcal{G}, \lambda}$ through the quotient $\mathcal{L}^+_n \mathcal{G}$. Let $\mathcal{L}\mathcal{G}_{\leq \mu}$ be the preimage of $\overline{\Gr}_{\mathcal{G}, \mu}$ under the quotient map $\mathcal{L} \mathcal{G} \rightarrow \mathcal{L} \mathcal{G} / \mathcal{L}^+ \mathcal{G} \cong \Gr_{\mathcal{G}}$, and let
$$\overline{\Gr}_{\mathcal{G}, \mu,n} = \mathcal{L}\mathcal{G}_{\leq \mu} \times^{\mathcal{L}^+ \mathcal{G}} \mathcal{L}^+_n \mathcal{G}.$$ Then there is a right $\mathcal{L}^+_n \mathcal{G}$-torsor $\phi_{\mathcal{G}, n} \colon \overline{\Gr}_{\mathcal{G}, \mu,n} \rightarrow \overline{\Gr}_{\mathcal{G}, \mu}$ such that
$$\overline{\Gr}_{\mathcal{G},\mu} \simtimes \overline{\Gr}_{\mathcal{G},\lambda}  \cong \overline{\Gr}_{\mathcal{G}, \mu,n} \times^{\mathcal{L}^+_n \mathcal{G}} \overline{\Gr}_{\mathcal{G}, \lambda}.$$ Thus it suffices to show that $\overline{\Gr}_{\mathcal{G}, \mu,n} \times_C \overline{\Gr}_{\mathcal{G}, \lambda}$ is irreducible.

Because $\mathcal{L}^+_n \mathcal{G}$ has geometrically irreducible fibers, then to prove that $\overline{\Gr}_{\mathcal{G}, \mu,n} \times_C \overline{\Gr}_{\mathcal{G}, \lambda}$ is irreducible it suffices to prove that $\overline{\Gr}_{\mathcal{G},\mu} \times_C \overline{\Gr}_{\mathcal{G},\lambda}$ is irreducible. Since $\overline{\Gr}_{\mathcal{G},\mu}$ and  $\overline{\Gr}_{\mathcal{G},\lambda}$ are flat over $C$ then $\overline{\Gr}_{\mathcal{G},\mu} \times_C \overline{\Gr}_{\mathcal{G},\lambda}$ is flat over $C$. Thus since $\restr{\overline{\Gr}_{\mathcal{G},\mu} \times_C \overline{\Gr}_{\mathcal{G},\lambda}}{C^\circ} = \Gr_{\leq \mu} \times \Gr_{\leq \lambda} \times C^\circ$ is irreducible it follows that $\overline{\Gr}_{\mathcal{G},\mu} \times_C \overline{\Gr}_{\mathcal{G},\lambda}$ is irreducible. Putting this all together, we have shown that $\overline{\Gr}^{\conv''}_{\mathcal{G}, \mu, \lambda} \cong (\overline{\Gr}_{\mathcal{G},\mu} \simtimes \overline{\Gr}_{\mathcal{G},\lambda})_{\text{red}}$.

By Corollary \ref{fibercor}, the fibers $(\overline{\Gr}_{\mathcal{G},\mu})_0$ and $(\overline{\Gr}_{\mathcal{G},\lambda})_0$ are reduced and Cohen--Macaulay. Hence the product $(\overline{\Gr}_{\mathcal{G},\mu})_0 \times (\overline{\Gr}_{\mathcal{G},\lambda})_0$ is also reduced and Cohen--Macaulay. As these two properties are local in the smooth topology, then $(\overline{\Gr}_{\mathcal{G},\mu})_0 \simtimes (\overline{\Gr}_{\mathcal{G},\lambda})_0$ is reduced and Cohen--Macaulay. Thus by the isomorphism $\overline{\Gr}^{\conv''}_{\mathcal{G}, \mu, \lambda} \cong (\overline{\Gr}_{\mathcal{G},\mu} \simtimes \overline{\Gr}_{\mathcal{G},\lambda})_{\text{red}}$ it follows that $(\overline{\Gr}^{\conv''}_{\mathcal{G}, \mu, \lambda})_0 \cong (\overline{\Gr}_{\mathcal{G},\mu})_0 \simtimes (\overline{\Gr}_{\mathcal{G},\lambda})_0$ is reduced and Cohen--Macaulay. As $\restr{\overline{\Gr}^{\conv''}_{\mathcal{G}, \mu, \lambda}}{C^\circ}$ is also Cohen--Macaulay then $\overline{\Gr}^{\conv''}_{\mathcal{G}, \mu, \lambda}$ is Cohen--Macaulay. 
\end{proof}

\begin{myproof1}[Proof of Theorem \ref{mainthm} (v)]
There is a morphism $m_{\mathcal{G}}'' \colon \Gr_{\mathcal{G}}^{\conv''} \rightarrow \Gr_{\mathcal{G}}$ which sends the element $(x,  \mathcal{E}_1, \mathcal{E}_2, \beta_1, \beta_2)$ to $(x, \mathcal{E}_2, \beta_1 \circ \beta_2)$. Over points in $C^\circ(k)$ the morphism $m_{\mathcal{G}}''$ is the convolution morphism $LG \times^{L^+G} \Gr \rightarrow \Gr$ and over $0$ the morphism $m_{\mathcal{G}}''$ is $m \colon LG \times^I \mathcal{F}\ell \rightarrow \mathcal{F}\ell$. By taking the fibers of $m_{\mathcal{G}}''$ over $C^\circ$ and $0$ we get the following diagram with Cartesian squares:

\begin{equation} \xymatrix{
(LG \times^{L^+G} \Gr) \times C^\circ \ar[r]^(.70){j^{\conv''}} \ar[d]^{} & \Gr_{\mathcal{G}}^{\conv''} \ar[d]^{m_{\mathcal{G}}''} & LG \times^I \mathcal{F}\ell \ar[l]_{i^{\conv''}} \ar[d]^{m} \\
\Gr \times C^\circ \ar[r]^(.55){j_{\mathcal{G}}}   & \Gr_{\mathcal{G}} & \mathcal{F}\ell \ar[l]_(.40){i_{\mathcal{G}}}  }
\end{equation}
Because the schemes $\overline{\Gr}^{\conv''}_{\mathcal{G}, \mu, \lambda} = (\overline{\Gr}_{\mathcal{G},\mu} \simtimes \overline{\Gr}_{\mathcal{G},\lambda})_{\text{red}}$ for $\mu$, $\lambda \in X_*(T)^+$ are Cohen--Macaulay, then by similar reasoning as in \cite[6.2]{modpGr} we can form the perverse sheaf
$$\mathcal{F}_{1,2}^\bullet := j_{\mathcal{G}, !*}(\mathcal{F}_1^\bullet \overset{L}{\boxtimes} \restr{\mathbb{F}_p[1]}{C^\circ}) \overset{\sim}{\boxtimes} j_{\mathcal{G}, !*}(\mathcal{F}_2^\bullet \overset{L}{\boxtimes} \restr{\mathbb{F}_p[1]}{C^\circ}) \in P_c^b(\Gr_{\mathcal{G}}^{\conv''}, \mathbb{F}_p).$$

To complete the proof it suffices to construct natural isomorphisms
\begin{enumerate}
\item[(i)] $Ri^{\conv'',*}[-1](\mathcal{F}_{1,2}^\bullet) \cong \mathcal{Z}(\mathcal{F}_1^\bullet) \overset{\sim}{\boxtimes} \mathcal{Z}(\mathcal{F}_1^\bullet)$,
\item[(ii)] $Rm_*(Ri^{\conv'',*}[-1](\mathcal{F}_{1,2}^\bullet)) \cong \mathcal{Z}( \mathcal{F}_1^\bullet * \mathcal{F}^\bullet_2).$
\end{enumerate}
Suppose $\mathcal{F}_1^\bullet$ is supported on $\Gr_{\leq \mu}$ and $\mathcal{F}_2^{\bullet}$ is supported on $\Gr_{\leq \lambda}$. As in the proof of Proposition \ref{CMProp}, fix an integer $n$ large enough so that $\mathcal{L}^+\mathcal{G}$ acts on $\overline{\Gr}_{\mathcal{G}, \lambda}$ through the quotient $\mathcal{L}^+_n \mathcal{G}$. Let $\phi_{\mathcal{G}, n} \colon \overline{\Gr}_{\mathcal{G}, \mu,n} \rightarrow \overline{\Gr}_{\mathcal{G}, \mu}$ be the right $\mathcal{L}^+_n \mathcal{G}$-torsor such that
$$\overline{\Gr}_{\mathcal{G},\mu} \simtimes \overline{\Gr}_{\mathcal{G},\lambda}  \cong \overline{\Gr}_{\mathcal{G}, \mu,n} \times^{\mathcal{L}^+_n \mathcal{G}} \overline{\Gr}_{\mathcal{G}, \lambda}.$$ Let $\phi_n^{\conv''} \colon (\overline{\Gr}_{\mathcal{G}, \mu,n} \times_C \overline{\Gr}_{\mathcal{G}, \lambda})_{\text{red}} \rightarrow \overline{\Gr}^{\conv''}_{\mathcal{G}, \mu, \lambda}$ be the resulting $\mathcal{L}^+_n \mathcal{G}$-torsor over $\overline{\Gr}^{\conv''}_{\mathcal{G}, \mu, \lambda}$. Over points in $C^\circ$ the map $\phi_{\mathcal{G},n} \colon \overline{\Gr}_{\mathcal{G}, \mu,n} \rightarrow \overline{\Gr}_{\mathcal{G}, \mu}$ restricts to an $L^nG$-torsor $p_n \colon \Gr_{\leq \mu, n} \rightarrow \Gr_{\leq \mu}$. By taking the fibers of $\phi_n^{\conv''}$ over $C^\circ$ and $0$ we get a diagram with Cartesian squares:
$$\xymatrix{
\Gr_{\leq \mu, n} \times \Gr_{\leq \lambda} \times C^\circ \ar[r]^(.50){j_n^{\conv''}} \ar[d] & (\overline{\Gr}_{\mathcal{G}, \mu,n} \times_C \overline{\Gr}_{\mathcal{G}, \lambda})_{\text{red}} \ar[d]^{\phi_n^{\conv''}} & (\overline{\Gr}_{\mathcal{G}, \mu,n})_{0} \times (\overline{\Gr}_{\mathcal{G}, \lambda})_0 \ar[d]^{\phi_{0,n}^{\conv''}} \ar[l]_(.45){i_n^{\conv''}} \\
\Gr_{\leq \mu} \simtimes \Gr_{\leq \lambda} \times C^{\circ} \ar[r]^(.6){j^{\conv''}} & \overline{\Gr}^{\conv''}_{\mathcal{G}, \mu, \lambda} & (\overline{\Gr}_{\mathcal{G},\mu})_{0} \simtimes (\overline{\Gr}_{\mathcal{G},\lambda})_{0} \ar[l]_(.55){i^{\conv''}}
}$$ An isomorphism as in (i) can be constructed in a manner similar to (\ref{eq4}) by pulling everything back to $(\overline{\Gr}_{\mathcal{G}, \mu,n})_{0} \times (\overline{\Gr}_{\mathcal{G}, \lambda})_0$. We leave the details to the reader. By using the left side of the above diagram one can also prove that
$$Rj^{\conv'',*}(\mathcal{F}_{1,2}^\bullet) \cong \mathcal{F}_1^\bullet \overset{\sim}{\boxtimes} \mathcal{F}_2^{\bullet} \overset{L}{\boxtimes} \restr{\mathbb{F}_p[1]}{C^\circ}.$$

By the proper base change theorem, to prove (ii) it suffices to construct a natural isomorphism
\begin{enumerate}
\item[(iii)] $Rm_{\mathcal{G},*}''(\mathcal{F}_{1,2}^\bullet) \cong j_{\mathcal{G},!*}((\mathcal{F}_1^\bullet * \mathcal{F}_2^\bullet) \overset{L}{\boxtimes} \restr{\mathbb{F}_p[1]}{C^\circ})$.
\end{enumerate}
By the description of $m_{\mathcal{G}}''$ over $C^\circ$ there is a natural isomorphism
$$R(\restr{m_{\mathcal{G}}''}{C^\circ})_*(Rj^{\conv'',*}(\mathcal{F}_{1,2}^\bullet)) \cong (\mathcal{F}_1^\bullet * \mathcal{F}_2^\bullet) \overset{L}{\boxtimes} \restr{\mathbb{F}_p[1]}{C^\circ}.$$ Thus the left side of (iii) is naturally an extension of $(\mathcal{F}_1^\bullet * \mathcal{F}_2^\bullet) \overset{L}{\boxtimes} \restr{\mathbb{F}_p[1]}{C^\circ}$, so we just need to show it is isomorphic to the intermediate extension. By semisimplicity we reduce to the case $\mathcal{F}_i^\bullet = \IC_{\mu_i}$ for $\mu_i \in X_*(T)^+$. Note that $m_{\mathcal{G}}''$ maps $\overline{\Gr}^{\conv''}_{\mathcal{G}, \mu_1, \mu_2}$ onto $\overline{\Gr}_{\mathcal{G}, \mu}$ where $\mu = \mu_1+\mu_2$. Moreover, $\IC_{\mu_1} * \IC_{\mu_2} \cong \IC_\mu$ by \cite[1.2]{modpGr}. Thus, $\mathcal{F}_{1,2}^\bullet$ and the right side of (iii) are shifted constant sheaves. 

Because the convolution morphism $\Gr_{\leq \mu_1} \simtimes \Gr_{\leq \mu_2} \rightarrow \Gr_{\leq \mu}$ is birational then the morphism $m_{\mathcal{G}}'' \colon \overline{\Gr}^{\conv''}_{\mathcal{G}, \mu_1, \mu_2} \rightarrow \overline{\Gr}_{\mathcal{G}, \mu}$ is also birational. We claim that $m_{\mathcal{G}}'' \colon \overline{\Gr}^{\conv''}_{\mathcal{G}, \mu_1, \mu_2} \rightarrow \overline{\Gr}_{\mathcal{G}, \mu}$ is projective. To prove this, note that $\overline{\Gr}_{\mathcal{G}, \mu}$ is projective over $C$ by \cite[5.5]{PappasZhu}. As $\Gr_{\mathcal{G}}^{\conv''} \cong \Gr_{\mathcal{G}} \simtimes \Gr_{\mathcal{G}} \cong \Gr_{\mathcal{G}} \times_C \Gr_{\mathcal{G}}$ (see, for example, \cite[1.2.14]{ZhuGra}), then $\overline{\Gr}^{\conv''}_{\mathcal{G}, \mu_1, \mu_2}$ is also projective over $C$. Thus $m_{\mathcal{G}}'' \colon \overline{\Gr}^{\conv''}_{\mathcal{G}, \mu_1, \mu_2} \rightarrow \overline{\Gr}_{\mathcal{G}, \mu}$ is  projective. Therefore, since $\overline{\Gr}_{\mathcal{G}, \mu}$ has pseudo-rational singularities and $\overline{\Gr}^{\conv''}_{\mathcal{G}, \mu_1, \mu_2}$ is Cohen--Macaulay then $Rm_{\mathcal{G},*}''(\mathcal{O}_{\overline{\Gr}^{\conv''}_{\mathcal{G}, \mu_1, \mu_2}}) \cong \mathcal{O}_{\overline{\Gr}_{\mathcal{G}, \mu}}$ by \cite[1.4]{ratsing}. Now (iii) follows by applying the Artin--Schreier sequence.
\end{myproof1}

\begin{rmrk} \label{assumption}
We have proved Theorem \ref{mainthm} under the hypothesis $G_{\text{der}}$ is simply connected and almost simple. We now explain how to remove the simple connectedness hypothesis using the same technique as in \cite[7.12]{modpGr}. The same idea is also used in \cite[3.3]{ZhuCoherence}.

Let $G$ be a connected reductive group over $k$ such that $G_{\text{der}}$ is almost simple. By \cite[3.1]{MilneShih} there exists a central extension $$1 \rightarrow N \rightarrow G' \rightarrow G \rightarrow 1$$ such that $G'_{\der}$ is simply connected and $N$ is a connected torus. Since $G_{\text{der}}$ is almost simple then so is $G'_{\text{der}}$. Let $T' \subset B' \subset G'$ be the maximal torus and Borel subgroup given by the preimages of $T$ and $B$, and let $I' \subset L^+G'$ be the preimage of $B'$ under the projection $L^+G' \rightarrow G'$. Let $\Gr_G$ and $\Gr_{G'}$ be the affine Grassmannians for $G$ and $G'$, respectively. Similarly, let $\mathcal{F}\ell_{G}$ and $\mathcal{F}\ell_{G'}$ be the Iwahori affine flag varieties.

Because $N$ is connected the map $X_*(T') \rightarrow X_*(T)$ is surjective. Hence the maps $\Gr_{G'} \rightarrow \Gr_{G}$ and $\mathcal{F}\ell_{G'} \rightarrow \mathcal{F}\ell_{G}$ are surjective. Let $\phi_1 \colon X_*(T')^+ \rightarrow X_*(T)^+$ be the induced surjection on dominant cocharacters and let $\phi_2 \colon \tilde{W}' \rightarrow \tilde{W}$ be the induced surjection on Iwahori--Weyl groups. Each connected component of $\Gr_{G'}$ maps onto its image in $\Gr_{G}$ via a universal homeomorphism by \cite[3.1]{UnivHomeo}. The same is true of the map $\mathcal{F}\ell_{G'} \rightarrow \mathcal{F}\ell_{G}$.

There is also a Bruhat--Tits group scheme $\mathcal{G}'$ over $C$ satisfying the conditions (\ref{isochoice2}) and equipped with a natural map $\mathcal{G}' \rightarrow \mathcal{G}$. This induces a map $\Gr_{\mathcal{G}'} \rightarrow \Gr_{\mathcal{G}}$. By restricting this map to $C^\circ$ and $0$ one sees that $\Gr_{\mathcal{G}'} \rightarrow \Gr_{\mathcal{G}}$ is surjective and maps each connected component of $\Gr_{\mathcal{G}'} $ into its image via a universal homeomorphism. Similarly, there are maps $\Gr_{\mathcal{G}'}^{\BD} \rightarrow \Gr_{\mathcal{G}}^{\BD}$, $\Gr_{\mathcal{G}'}^{\conv} \rightarrow \Gr_{\mathcal{G}}^{\conv}$, etc., which are surjective and restrict to universal homeomorphisms on connected components.

All of the diagrams of $C$-schemes we used in Sections \ref{S2} and \ref{S3} to prove Theorem 1.1 for $G'$ intertwine with the corresponding diagrams for $G$. For example, we have the following commutative diagram:
$$\xymatrix{
\Gr_{\mathcal{G}'} \ar[r] \ar[d]^{\pi_{\mathcal{G}'}} & \Gr_{\mathcal{G}} \ar[d]^{\pi_{\mathcal{G}}} \\
\Gr_{\underline{G}'} \ar[r] & \Gr_{\underline{G}}
}$$
Thus, by the topological invariance of the \'{e}tale site \cite[04DY]{stacks-project}, our arguments in Sections \ref{S2} and \ref{S3} can be used to simultaneously prove Theorem 1.1 for both $G'$ and $G$. 
\end{rmrk}

\section{Applications} \label{S4}
\subsection{The function-sheaf correspondence}
Let $X_0$ be a separated scheme of finite type over $\mathbb{F}_q$ and let $\mathcal{F}^\bullet _0 \in D_c^b(X_{0,\text{\'{e}t}}, \mathbb{F}_p)$. Fix an embedding of $\mathbb{F}_q$ into an algebraic closure $\overline{\mathbb{F}}_q$ and let $F^* \in \Gal(\overline{\mathbb{F}}_q/\mathbb{F}_q)$ be the inverse of the map which sends $\alpha \mapsto \alpha^q$. Let $X = X_0 \times_{\Spec(\mathbb{F}_q)} \Spec(\overline{\mathbb{F}}_q)$ and let $\mathcal{F}^\bullet$ be the pullback of $\mathcal{F}_0^\bullet$ to $X$. For $x \in X_0(\mathbb{F}_q)$ let $\mathcal{F}^\bullet_x$ be the pullback of $\mathcal{F}^\bullet$ along the composition $\Spec(\overline{\mathbb{F}}_q) \rightarrow \Spec(\mathbb{F}_q) \xrightarrow{x} X_0$. Then each $H^i(\mathcal{F}^\bullet_x)$ is a representation of $\Gal(\overline{\mathbb{F}}_q/\mathbb{F}_q)$ and it makes sense to take the trace $\Tr(F^*, H^i(\mathcal{F}^\bullet_x))$. 

We form a function $\Tr(\mathcal{F}^\bullet_0) \colon X_0(\mathbb{F}_q) \rightarrow \mathbb{F}_p$ by setting
$$\Tr(\mathcal{F}^\bullet_0)(x) = \sum_i (-1)^i \Tr(F^*, H^i(\mathcal{F}^\bullet_x)).$$ See also \cite[Ch. 2, \S 1]{SGA412} for more information on the construction of the function $\Tr(\mathcal{F}^\bullet_0)$. 
As in the case of $\overline{\mathbb{Q}}_{\ell}$-coefficients, we have:

\begin{thm} \label{fslemma} Let  $X_0$ and $Y_0$ be separated schemes of finite type over $\mathbb{F}_q$.
\leavevmode
\begin{enumerate}[{\normalfont (i)}]
\item Let $\mathcal{F}^\bullet_0 \in D_c^b(X_{0, \textnormal{\'{e}t}}, \mathbb{F}_p)$ and $\mathcal{G}^\bullet_0 \in D_c^b(Y_{0, \textnormal{\'{e}t}}, \mathbb{F}_p)$. If $x \in (X_0 \times_{\Spec(\mathbb{F}_q)} Y_0)(\mathbb{F}_q)$ has images $p_1(x) \in X_0(\mathbb{F}_q)$ and $p_2(x) \in Y_0(\mathbb{F}_q)$ then $$\Tr(\mathcal{F}^\bullet_0 \overset{L}{\boxtimes} \mathcal{G}^\bullet_0)(x) = \Tr(\mathcal{F}^\bullet_0)(p_1(x)) \Tr(\mathcal{G}^\bullet_0)(p_2(x)).$$ 
\item Let $f \colon Y_0 \rightarrow X_0$ be a morphism. If $\mathcal{F}^\bullet_0 \in D_c^b(X_{0, \textnormal{\'{e}t}}, \mathbb{F}_p)$ and $y \in Y_0(\mathbb{F}_q)$ then $$\Tr(Rf^* (\mathcal{F}^\bullet_0))(y) = \Tr(\mathcal{F}^\bullet_0)(f(y)).$$
\item If $\mathcal{F}^\bullet_0 \in D_c^b(X_{0, \textnormal{\'{e}t}}, \mathbb{F}_p)$ then $$\displaystyle\sum_{x \in X_0(\mathbb{F}_q)} \Tr(\mathcal{F}^\bullet_0)(x) = \sum_i (-1)^i \Tr(F^*, R^i \Gamma_c (X,\mathcal{F}^\bullet)).$$
\end{enumerate}
\end{thm}

\begin{proof}
Parts (i) and (ii) are immediate from the definitions, and part (iii) is \cite[4.1]{SGA412}.  
\end{proof}

\subsection{Perverse $\mathbb{F}_p$-sheaves over finite fields} For the rest of Section \ref{S4} we assume $G$ is a split connected reductive group defined over $\mathbb{F}_q$ and that $G_{\der}$ is absolutely almost simple. Let $\Gr_{\mathbb{F}_q}$ and $\mathcal{F}\ell_{\mathbb{F}_q}$ denote the affine Grassmannian and affine flag variety viewed as ind-schemes over $\mathbb{F}_q$. While we restricted to the case of an algebraically closed ground field in \cite{modpGr}, our constructions also work over an arbitrary perfect field of characteristic $p > 0$. In particular, we can construct the categories $P_{L^+G}(\Gr_{\mathbb{F}_q}, \mathbb{F}_p)$ and $P_I(\mathcal{F}\ell_{\mathbb{F}_q}, \mathbb{F}_p)$. The main difference when working over $\mathbb{F}_q$ is that there are more objects that are simple. In particular, \cite[3.18]{modpGr} needs to be revised in this setting, as there are non-trivial simple \'{e}tale local systems on $\Spec(\mathbb{F}_q)$. 

As in the case of $\overline{\mathbb{Q}}_\ell$-coefficients in \cite[5.6]{ZhuGra}, we will restrict ourselves to a certain subcategory of $P_{L^+G}(\Gr_{\mathbb{F}_q}, \mathbb{F}_p)$ consisting of normalized perverse sheaves. More precisely, let $\mathcal{L}$ be the  \'{e}tale local system on $\Spec(\mathbb{F}_q)$ corresponding to the representation $\Gal(\overline{\mathbb{F}}_q/ \mathbb{F}_q) \rightarrow \GL_1(\mathbb{F}_p)$ which sends $F^*$ to $-1$. If $X_0$ is a scheme over $\mathbb{F}_q$ then we can also view $\mathcal{L}$ as a local system on $X_0$ by pulling back along $X_0 \rightarrow \Spec(\mathbb{F}_q)$. 

Let $\mu \in X_*(T)^+$ be such that $\dim \Gr_{\leq \mu}$ has parity $p(\mu) \in \{0, 1\}$. Because \cite[1.7]{modpGr} also holds when $k$ is perfect (with the same proof), then $\IC_\mu$ is isomorphic to the shifted constant sheaf $\mathbb{F}_p[\dim \Gr_{\leq \mu}]$ supported on $\Gr_{\mathbb{F}_q, \leq \mu}$. We define the normalized $\IC$ complex
$$\IC_\mu^{\N} : = \IC_\mu \overset{L}{\otimes} \: \mathcal{L}^{\otimes p(\mu)} \in P_{L^+G}(\Gr_{\mathbb{F}_q}, \mathbb{F}_p).$$
Let $P_{L^+G}(\Gr_{\mathbb{F}_q}, \mathbb{F}_p)^{\N} \subset P_{L^+G}(\Gr_{\mathbb{F}_q}, \mathbb{F}_p)$ be the Serre subcategory consisting of perverse sheaves whose simple subquotients are all the form $\IC_\mu^{\N}$ for $\mu \in X_*(T)^+$. 

We claim the subcategory $P_{L^+G}(\Gr_{\mathbb{F}_q}, \mathbb{F}_p)^{\N}$ is monoidal. To prove this, we first note that the identity $\IC_{\mu_1} * \IC_{\mu_2} = \IC_{\mu_1+\mu_2}$ in \cite[1.2]{modpGr} also holds over $\mathbb{F}_q$. Indeed, this identity is derived from the result of Kov{\'a}cs \cite[1.4]{ratsing} which is independent of the ground field. From this fact and the projection formula \cite[0B54]{stacks-project} it follows that $\IC_{\mu_1}^{\N} * \IC_{\mu_2}^{\N} = \IC_{\mu_1+\mu_2}^{\N}$. 

Using the same arguments as in \cite{modpGr} one can show that $P_{L^+G}(\Gr_{\mathbb{F}_q}, \mathbb{F}_p)^{\N}$ is a symmetric monoidal category. For $\mathcal{F}^\bullet \in   P_{L^+G}(\Gr_{\mathbb{F}_q}, \mathbb{F}_p)$ let $\mathcal{F}_{\Gr}^\bullet \in P_{L^+G}(\Gr, \mathbb{F}_p)$ be its pullback to $\Gr$. Then one can also show that 
$$P_{L^+G}(\Gr_{\mathbb{F}_q}, \mathbb{F}_p)^{\N} \rightarrow \Vect_{\mathbb{F}_p}, \quad \mathcal{F}^\bullet \mapsto \bigoplus_i R^i\Gamma(\mathcal{F}^{\bullet}_{\Gr})
$$ is an exact, faithful, tensor functor. 

For $\mu \in X_*(T)^+$ let $$\mathcal{Z}_{\mathcal{A}(\mu)}^{\N} := \mathcal{Z}_{\mathcal{A}(\mu)} \overset{L}{\otimes} \: \mathcal{L}^{\otimes p(\mu)} \in P_I(\mathcal{F}\ell_{\mathbb{F}_q}, \mathbb{F}_p).$$ Let $P_{L^+G}(\Gr_{\mathbb{F}_q}, \mathbb{F}_p)^{\N, \text{ss}} \subset P_{L^+G}(\Gr_{\mathbb{F}_q}, \mathbb{F}_p)^{\N}$ be the tensor subcategory consisting of semisimple objects. If $p > 2$ the arguments in this paper also work over $\mathbb{F}_q$ and give rise to a functor $$\mathcal{Z} \colon P_{L^+G}(\Gr_{\mathbb{F}_q}, \mathbb{F}_p)^{\N, \text{ss}}  \rightarrow P_I(\mathcal{F}\ell_{\mathbb{F}_q}, \mathbb{F}_p)$$ which satisfies all parts of Theorem \ref{mainthm} and such that $\mathcal{Z}(\IC_\mu^{\N}) \cong \mathcal{Z}_{\mathcal{A}(\mu)}^{\N}.$

\begin{rmrk} The functor $P_{L^+G}(\Gr_{\mathbb{F}_q}, \mathbb{F}_p)^{\N} \rightarrow P_{L^+G}(\Gr, \mathbb{F}_p)$ induced by pullback is faithful and identifies the simple objects in these two categories, but it is not an equivalence of categories in general. The issue is that due to the failure of smooth base change, the group of extensions between two objects depends on the ground field (see the proof of \cite[6.14]{modpGr}). 
\end{rmrk}

\subsection{Proofs of applications to Hecke algebras}
In this section we prove Theorem \ref{mainapp} and Corollary \ref{maincor}. We fix an isomorphism $\mathbb{F}_q  [\![ t ]\!] \cong \mathcal{O}_E$. Recall that if $H \subset G(E)$ is a compact open subgroup then the multiplication on $\mathcal{H}_H$ is defined by
$$(f*g)(x) = \sum_{y \in G(E)/H} f(xy) g(y^{-1}).$$
We now verify that the function-sheaf correspondence respects the convolution of perverse sheaves and functions.

\begin{lem} \label{applem1} If $\mathcal{F}_1^\bullet$, $\mathcal{F}_2^{\bullet} \in P_{L^+G}(\Gr_{\mathbb{F}_q}, \mathbb{F}_p)$ then 
$$\Tr(\mathcal{F}_1^\bullet * \mathcal{F}_2^{\bullet}) = \Tr(\mathcal{F}_1^\bullet) * \Tr(\mathcal{F}_2^{\bullet}) \in \mathcal{H}_K.$$
Similarly, if $\mathcal{F}_1^\bullet$, $\mathcal{F}_2^{\bullet} \in P_I(\mathcal{F}\ell_{\mathbb{F}_q}, \mathbb{F}_p)$ then
$$\Tr(\mathcal{F}_1^\bullet * \mathcal{F}_2^{\bullet}) = \Tr(\mathcal{F}_1^\bullet) * \Tr(\mathcal{F}_2^{\bullet}) \in \mathcal{H}_I.$$
\end{lem}

\begin{proof}
We will only consider the case $\mathcal{F}_1^\bullet$, $\mathcal{F}_2^{\bullet} \in P_I(\mathcal{F}\ell_{\mathbb{F}_q}, \mathbb{F}_p)$; the other case can be handled by similar methods. The proof is essentially the same as that of \cite[5.6.1]{ZhuGra} in the case of  $\overline{\mathbb{Q}}_\ell$-coefficients, but because this lemma is important for our applications we will reproduce the details. Let $m \colon LG \times^{I} \mathcal{F}\ell \rightarrow \mathcal{F}\ell$ be the multiplication map and let $x \in \mathcal{F}\ell_{\mathbb{F}_q}(\mathbb{F}_q) = G(E)/I$. There is a natural identification $m^{-1}(x)(\mathbb{F}_q) = \{(xy, y^{-1}) \: : \: y \in G(E)/I\}$. By Theorem \ref{fslemma} (iii) and the proper base change theorem,
$$\Tr( \mathcal{F}_1^\bullet * \mathcal{F}_2^{\bullet})(x) = \sum_{y \in G(E)/I} \Tr( \mathcal{F}_1^\bullet \overset{\sim}{\boxtimes} \mathcal{F}_2^{\bullet})(xy,y^{-1}).$$ Now viewing $xy$ and $y^{-1}$ as elements of $G(E)/I$ and using Theorem \ref{fslemma} (i) and (ii) we have
$$ \Tr( \mathcal{F}_1^\bullet \overset{\sim}{\boxtimes} \mathcal{F}_2^{\bullet})(xy,y^{-1}) = \Tr(\mathcal{F}_1^\bullet)(xy) \Tr(\mathcal{F}_2^{\bullet})(y^{-1}).$$
\end{proof}

\begin{lem} \label{applem2}
Let $\mathcal{F}^\bullet \in P_I(\mathcal{F}\ell_{\mathbb{F}_q}, \mathbb{F}_p)$. Then
$$\Tr(R \pi_! (\mathcal{F}^\bullet)) = \Tr(\mathcal{F}^\bullet) * \mathds{1}_K \in \mathcal{H}_K.$$
\end{lem}

\begin{proof}
If $x \in \Gr_{\mathbb{F}_q}(\mathbb{F}_q)$ then $\pi^{-1}(x)(\mathbb{F}_q) = \{xy \: : \: y \in K/I \}$. We claim that
$$\Tr(R \pi_! (\mathcal{F}^\bullet))(x)  = \sum_{y \in K/I} \Tr(\mathcal{F}^\bullet)(xy) = \sum_{y \in G(E)/I}  \Tr(\mathcal{F}^\bullet)(xy)  \mathds{1}_K(y^{-1}) = (\Tr(\mathcal{F}^\bullet) * \mathds{1}_K)(x) .$$ The first equality follows from Theorem \ref{fslemma} (iii) and the proper base change theorem. The other two equalities follow from the definitions. 
\end{proof}

Since $\IC_{\mu_1}^{\N} * \IC_{\mu_2}^{\N} = \IC_{\mu_1+\mu_2}^{\N}$ there is a natural isomorphism of $\mathbb{F}_p$-algebras $$f \colon \mathbb{F}_p[X_*(T)^+] \rightarrow K_0(P_{L^+G}(\Gr_{\mathbb{F}_q}, \mathbb{F}_p)^{\N}) \otimes \mathbb{F}_p, \quad \mu \mapsto [\IC_\mu^{\N}].$$ The next proposition shows that this isomorphism is compatible with the mod $p$ Satake isomorphism.

\begin{prop} \label{applem3}
The composition 
$$\mathbb{F}_p[X_*(T)^+] \xrightarrow{f} K_0(P_{L^+G}(\Gr_{\mathbb{F}_q}, \mathbb{F}_p)^{\N}) \otimes \mathbb{F}_p \xrightarrow{\Tr} \mathcal{H}_K$$ is the inverse of the mod $p$ Satake isomorphism $\mathcal{S}$.
\end{prop}

\begin{proof}
When $G_{\text{der}}$ is simply connected, Herzig \cite[5.1]{modpclass} has computed $$\mathcal{S}^{-1} \left( \sum_{\lambda \leq \mu} \mathds{1}_{\lambda} \right) = \mu \in \mathbb{F}_p[X_*(T)^+].$$ While Herzig works with anti-dominant cocharacters, the formula above can be obtained from Herzig's formula by multiplying by the longest element of the Weyl group (see the proof of \cite[1.3]{modpGr}). Herzig only uses the hypothesis that $G_{\text{der}}$ is simply connected to reduce to the case where the weight (a representation of $G(\mathbb{F}_q)$) is trivial. As the weight is trivial in our situation, the above formula is valid for any $G$. Since $\IC_\mu$ is a constant sheaf supported on $\Gr_{\mathbb{F}_q, \leq \mu}$ then the lemma now follows from our choice of normalized IC complexes $\IC_{\mu}^{\N}$.
\end{proof}

\begin{rmrk}
In \cite{modpGr} we worked over an algebraically closed field and described a natural map of $\mathbb{F}_p$-vector spaces $K_0(P_{L^+G}(\Gr, \mathbb{F}_p)) \otimes \mathbb{F}_p \rightarrow \mathcal{H}_K$ which we proved is the inverse of the mod $p$ Satake isomorphism, and hence also an isomorphism of $\mathbb{F}_p$-algebras. It is possible to work over an algebraically closed field because, for $\mathcal{F}^\bullet \in P_{L^+G}(\Gr_{\mathbb{F}_q}, \mathbb{F}_p)^{\N}$, $\Tr(\mathcal{F}^\bullet)$ essentially counts the dimensions of the stalks of $\mathcal{F}^\bullet$. However, one advantage of working over $\mathbb{F}_q$ is that Lemma \ref{applem1} allows us to prove the existence of an isomorphism of $\mathbb{F}_p$-algebras $\mathbb{F}_p[X_*(T)^+] \cong \mathcal{H}_K$ without using the existence of the mod $p$ Satake isomorphism.
\end{rmrk}

\begin{myproof1}[Proof of Theorem \ref{mainapp} and Corollary \ref{maincor}] Theorem \ref{mainapp} follows immediately from Theorem \ref{mainthm} and Lemmas \ref{applem1}, \ref{applem2}. By \cite{Ollivier2014} the central integral Bernstein elements are uniquely determined by the identity $\mathcal{B} = \mathcal{C}^{-1} \circ \mathcal{S}^{-1}$, so Corollary \ref{maincor} follows from Theorem \ref{mainapp} and Proposition \ref{applem3}.
\end{myproof1}

\end{document}